\documentclass[11pt,a4paper,reqno]{amsart}
\usepackage{amsmath, amssymb, latexsym, enumerate,  graphicx,tikz, stmaryrd, eufrak,enumitem}

\usepackage{thmtools}

\usepackage[pdftex]{hyperref}
\usepackage{cleveref}
\usepackage{bbm}
\usepackage{cases}
\usepackage{array}
\usepackage{tikz-cd}
\usepackage[all]{xy}

\usepackage[hmarginratio={1:1},vmarginratio={1:1},lmargin=80.0pt,tmargin=90.0pt]{geometry}

\usepackage{tikz}
\tikzset{anchorbase/.style={baseline={([yshift=-0.5ex]current bounding box.center)}}}
\usetikzlibrary{decorations.markings}
\usetikzlibrary{decorations.pathreplacing}
\usetikzlibrary{arrows,shapes,positioning,backgrounds}
\tikzstyle directed=[postaction={decorate,decoration={markings,
    mark=at position #1 with {\arrow{>}}}}]
\tikzstyle rdirected=[postaction={decorate,decoration={markings,
    mark=at position #1 with {\arrow{<}}}}]
    
\usepackage{graphicx}
\usepackage{nicefrac}

 \newlength{\baseunit}               
 \newcount{\numlines}                
\newtheorem{theorem}[subsubsection]{Theorem}
\newtheorem{lemma}[theorem]{Lemma}
\newtheorem{prop}[theorem]{Proposition}
\newtheorem{corollary}[subsubsection]{Corollary}
\newtheorem{conjecture}[theorem]{Conjecture}

\theoremstyle{definition}
\newtheorem{definition}[subsubsection]{Definition}

\newtheorem{remark}[theorem]{Remark}

\newtheorem{example}[subsubsection]{Example}

\newtheorem{question}[theorem]{Question}

\newtheorem{thmA}{Theorem}

\newtheorem{conjA}[thmA]{Conjecture}

\newcommand{\mG}{\mathbb{G}}

\newcommand{\cA}{\mathcal{A}}
\newcommand{\cC}{\mathcal{C}}
\numberwithin{equation}{section}

\newcommand{\ev}{\mathrm{ev}}
\newcommand{\co}{\mathrm{co}}
\newcommand{\im}{\mathrm{im}}

\newcommand{\bC}{\mathcal{C}}

\newcommand{\bD}{\mathcal{D}}
\newcommand{\bT}{\mathcal{C}}

\newcommand{\cU}{\mathcal{U}}

\newcommand{\cS}{\mathcal{S}}

\newcommand{\cD}{\mathcal{D}}

\newcommand{\Ver}{\mathsf{Ver}}

\newcommand{\Mod}{\mathsf{Mod}}
\newcommand{\Rep}{\mathsf{Rep}}

\newcommand{\Tens}{\mathsf{Tens}}

\newcommand{\unit}{{\mathbbm{1}}}
\newcommand{\tto}{\twoheadrightarrow}
\newcommand{\cO}{\mathcal{O}}

\newcommand{\mN}{\mathbb{N}}
\newcommand{\mZ}{\mathbb{Z}}
\newcommand{\mA}{\mathbb{A}}

\newcommand{\mC}{\mathbb{C}}
\newcommand{\mR}{\mathbb{R}}

\newcommand{\End}{\mathrm{End}}

\newcommand{\Ext}{\mathrm{Ext}}

\newcommand{\Hom}{\mathrm{Hom}}

\newcommand{\Sym}{\mathrm{Sym}}
\newcommand{\id}{\mathrm{id}}

\newcommand{\Ind}{\mathrm{Ind}}

\newcommand{\inv}{\mathrm{inv}}
\newcommand{\Vecc}{\mathsf{Vec}}
\newcommand{\sVecc}{\mathsf{sVec}}
\newcommand{\sVec}{\mathsf{sVec}}

\newcommand{\Tilt}{\mathsf{Tilt}}

\newcommand{\Alg}{\mathsf{Alg}}

\newcommand{\PTann}{\mathsf{PTann}}
\newcommand{\Tann}{\mathsf{Tann}}
\newcommand{\MdGr}{\mathsf{MdGr}}
\newcommand{\FrEx}{\mathsf{FrEx}}
\newcommand{\DTann}{\mathcal{D}\mbox{-}\Tann}

\newcommand{\Tr}{\mathrm{Tr}}
\newcommand{\Id}{\mathrm{Id}}
\newcommand{\Aut}{\mathrm{Aut}}

\newcommand{\sgd}{\mathsf{sgd}}
\newcommand{\gd}{\mathsf{gd}}
\newcommand{\gr}{\mathrm{gr}}
\newcommand{\Fr}{\mathrm{Fr}}
\newcommand{\FPdim}{\mathrm{FPdim}}

\newcommand{\uAut}{\underline{\mathrm{Aut}}^{\otimes}}

\newcommand{\fk}{\mathbf{k}}
\newcommand{\bk}{\mathbf{k}}
\newcommand{\mF}{\mathbb{F}}
\newcommand{\HS}{\mathsf{HS}}

\newcommand{\Dim}{\mathrm{Dim}}

\begin{document}
\title{Incompressible tensor categories}
\author{Kevin Coulembier, Pavel Etingof and Victor Ostrik}

\date{\today}

\keywords{tensor categories, incompressible categories, higher Verlinde categories, \'etale algebras, tannakian reconstruction}

\begin{abstract} A symmetric tensor category $\mathcal D$ over an algebraically closed field $\bold k$ is called {\bf incompressible} if its objects have finite length ($\cD$ is pretannakian) and every tensor functor out of $\mathcal D$ is an embedding of a tensor subcategory. E.g., the categories $\Vecc$, $\sVecc$  
of vector and supervector spaces are incompressible. 
Moreover, by Deligne's theorem \cite{Del02}, if ${\rm char}(\bold k)=0$ then any tensor category of moderate growth uniquely fibres over $\sVecc$. This implies that $\Vecc$, $\sVecc$ are the only incompressible categories over $\bk$ in this class, and perhaps altogether, as we expect that all incompressible categories have moderate growth. 

Similarly, in characteristic $p>0$, we also have incompressible 
Verlinde categories $\Ver_p,\Ver_p^+$, 
and by \cite{CEO} any Frobenius exact category of moderate growth uniquely fibres over $\Ver_p$, meaning that, in this class, the above categories are the only incompressible ones. More generally, the Verlinde categories $\Ver_{p^n}$, $\Ver_{p^n}^+$, $n\le \infty$ introduced in \cite{BEO,AbEnv} are incompressible, and a key conjecture is that every tensor category of moderate growth uniquely fibres over $\Ver_{p^\infty}$. This would make the above the only incompressible categories in this class (and perhaps altogether).

We prove a part of this conjecture, showing that every tensor category of moderate growth fibres over an incompressible one. So it remains to understand incompressible categories, and we prove several results in this direction. Namely, let $\mathcal D$-Tann be the category of tensor categories that fibre over $\mathcal D$. Then we say that $\mathcal D$ is {\bf subterminal} if it is a terminal object of $\mathcal D$-Tann (i.e., fibre functors to $\mathcal D$ are unique when exist), and that $\mathcal D$ is a {\bf Bezrukavnikov category} 
if $\mathcal D$-Tann is closed under taking images of tensor functors (=quotient categories). Clearly, a subterminal Bezrukavnikov category 
is incompressible, and we conjecture that the converse also holds; for instance, tensor subcategories of $\Ver_p$ are known to be subterminal. We prove that furthermore they are Bezrukavnikov categories, generalizing the result of Bezrukavnikov \cite{B} in the case of 
$\Vecc$, $\sVecc$. Finally, we show that a tensor subcategory of a finite incompressible category is incompressible. 

We also find intrinsic sufficient conditions for incompressibility and subterminality. Namely, $\mathcal D$ is called {\bf maximally nilpotent} if the growth rate of the length of the symmetric powers of every object is the minimal one theoretically possible. We show that a maximally nilpotent category 
is incompressible, and also subterminal if it satisfies an additional 
{\bf geometric reductivity} condition (for every morphism $X\tto\unit$ in $\cD$, there exists $n>0$ for which $\Sym^n X\tto\unit$ is split). Then we verify these conditions for the category $\Ver_{2^n}$, thereby showing that it is subterminal. 

\end{abstract}

\maketitle

\centerline{\bf To Roman Bezrukavnikov on his 50th birthday with admiration} 

\tableofcontents

\section{Introduction}

\subsection{Deligne's theorem and its analogue in positive characteristic} 
Let $\fk$ be an algebraically closed field. We study the structure of the category of pretannakian categories over $\bk$ (symmetric tensor categories with objects of finite length), see \cite{DM, EGNO}. In characteristic zero, this is governed by Deligne's classical results in \cite{Del90, Del02, Del07}, with central result that a tensor category admits a tensor functor to the category $\sVec_{\bk}$ of supervector spaces (and is thus the representation category of a supergroup) if and only if it is of moderate growth. The latter condition demands that for any given object, the growth of the length of its tensor powers is bounded by an exponential function.

For $\bk$ of characteristic $p>0$, the subject has developed rapidly in the last few years, see for instance \cite{BE, BEO, CEO, Tann, EG1, EG2, EOf, Os}. An analogue of Deligne's theorem for Frobenius exact  tensor categories (meaning the Frobenius functor from \cite{Tann, EOf} is exact) was obtained in \cite{CEO}, where one has to replace the category of supervector spaces with the larger {\it Verlinde category} $\Ver_p$.

\subsection{Incompressible categories} 
During these developments,
 the notion of an {\bf incompressible tensor category} has emerged. A pretannakian category is said to be incompressible if every tensor functor out of it is an embedding of a tensor subcategory. Deligne's results imply that in characteristic zero, the only incompressible categories of moderate growth are $\sVec_{\bk}$ and its unique proper tensor subcategory $\Vecc_{\bk}$. By sharp contrast, in characteristic $p>0$, in \cite{BEO}, see also \cite{BE, AbEnv}, an infinite chain 
\begin{equation}\label{VerChain}
\Ver_p\;\subset\;\Ver_{p^2}\;\subset\; \Ver_{p^3}\;\subset\;\cdots\;\subset\; \Ver_{p^n}\;\subset\;\cdots
\end{equation}
of incompressible tensor categories of moderate growth was constructed. Moreover, it was conjectured in \cite{BEO} that any tensor category of moderate growth (not necessarily Frobenius exact) admits a tensor functor to the union $\Ver_{p^\infty}=\cup_n\Ver_{p^n}$.

\subsection{Tensor functors to incompressible categories} 
Our first main result is the following weak form of this conjecture. 

\begin{thmA}\label{ThmIncom}
For $\fk$ of arbitrary characteristic, each pretannakian category of moderate growth admits a tensor functor to {\em some} incompressible category of moderate growth.
\end{thmA}

Thus Conjecture~1.4 in \cite{BEO} is now equivalent to the following (a priori) weaker conjecture: 

\begin{conjA} 
If $\mathrm{char}(\fk)=p>0$, then  every incompressible category of moderate growth over $\fk$ is a tensor subcategory of $\Ver_{p^\infty}$.
\end{conjA}

Note that the tensor subcategories of $\Ver_{p^\infty}$ are classified, see~\cite[Corollary~4.61]{BEO}. For instance, if $p>2$, they are given by 
$$\Ver_{p^n} \quad\mbox{and}\quad \Ver_{p^n}^+, \qquad\mbox{with $n\in\mN\cup\{\infty\}$,}$$
where (for $n=0$) we set $\Ver_{1}=\sVec_{\bk}$ and $\Ver^+_1=\Vecc_{\bk}$.

We conjecture that Theorem~\ref{ThmIncom} extends to an if and only if statement.
\begin{conjA}\label{ConjConj}
A pretannakian category admits a tensor functor to an incompressible category if and only if it is of moderate growth. Equivalently, all incompressible pretannakian categories are of moderate growth.
\end{conjA}
Working towards this conjecture, we prove in Corollary~\ref{CorEta} that Deligne's universal monoidal category $(\Rep S)_{t}$ (which is not of moderate growth) does not admit a tensor functor to an incompressible category.

Denote by $\MdGr_{\bk}$ the category of tensor categories of moderate growth, a full subcategory of the category $\PTann_{\bk}$ of pretannakian categories, which is in turn a full subcategory of the category $\Tens_{\bk}$ of all (symmetric) tensor categories over~$\bk$. When $\bk$ is clear from context, we will sometimes leave out reference to it.
Deligne's theorem implies that in characteristic zero, 
the category $\MdGr_{\fk}$ has a terminal object, $\sVec_{\bk}$. In contrast, our next result is 

\begin{thmA}\label{ThmDel} (Proposition~\ref{PropFilt})
If $\mathrm{char}(\fk)=0$ then the categories $\PTann_{\fk}$ and $\Tens_{\fk}$ are not filtered, so in particular have no terminal object.
\end{thmA}

\subsection{The category $\cD\mbox{-}\Tann$}\label{SecDTann}
Before formulating the rest of our results, let us introduce some notation. For any $\cD\in \PTann_{\bk}$, let $\cD\mbox{-}\Tann\subset \PTann_{\bk}$ denote the full subcategory of $\PTann_{\bk}$ consisting of tensor categories which admit a tensor functor to $\cD$. By Deligne's tannakian reconstruction principle from \cite{Del90}, $\DTann$ comprises precisely the (appropriately defined) categories of representations of affine group schemes in $\cD$. For example, $\Vecc_{\bk}\mbox{-}\Tann=\Tann_{\bk}$ is the category of tannakian categories and $\sVec_{\bk}\mbox{-}\Tann={\rm s}\Tann_{\bk}$ is the category of supertannakian categories. 

In terms of this notation, the main result of \cite{Del02} states that if $\bk$ has characteristic zero, then $\MdGr_{\bk}=\sVec_{\bk}\mbox{-}\Tann$. Likewise the main result of \cite{CEO} states that if $\bk$ has characteristic $p>0$, then $\Ver_p\mbox{-}\Tann$ is the full subcategory $\FrEx_{\bk}\subset \MdGr_{\bk}$ consisting of Frobenius exact categories, and \cite[Conjecture~1.4]{BEO} states that $\MdGr_{\bk}=\Ver_{p^\infty}\mbox{-}\Tann$. Finally, Theorem~\ref{ThmIncom} states that $\MdGr_{\bk}$ is the union of $\DTann$ for $\cD$ ranging over all incompressible categories $\mathcal D$ of moderate growth.

\subsection{Subterminal and Bezrukavnikov categories}
The classical incompressible categories $\Vecc_{\bk}$ and $\sVec_{\bk}$ satisfy some further important properties. Namely, let us say 
that a pretannakian category $\cD$ is {\bf subterminal} if 
it is the terminal object of $\cD\mbox{-}\Tann$, i.e., if 
tensor functors to $\cD$ are unique if exist (equivalently, $\cD$ is a subterminal object of $\Tens_{\bk}$). It is proved in \cite{DM, Del02, Tann} that $\Vecc_{\bk}$ and $\sVec_{\bk}$ are subterminal. 
 Additionally, in \cite{B} it is proved that every surjective tensor functor out of a (super)tannakian category lands in a (super)tannakian category. Applying this notion to an arbitrary pretannakian category $\cD$, we say that $\cD$ is a {\bf Bezrukavnikov category} if $\DTann$ is closed under taking images of tensor functors (i.e., quotient tensor categories).
In characteristic zero, we will show in Proposition \ref{incoo}(3) 
that the only subterminal pretannakian categories are precisely $\Vecc_{\bk}$ and $\sVec_{\bk}$. 

It is easy to show that if a pretannakian category is both subterminal and Bezrukavnikov then it must be incompressible, and as noted above, 
the converse is expected to hold at least in characteristic zero, by Conjecture~\ref{ConjConj}. Whether incompressible categories in positive characteristic (such as $\Ver_{p^n}$) must be subterminal and Bezrukavnikov are interesting open questions which we 
answer in special cases. Our next result 
shows that these questions are again intimately related to the structure of the category of tensor categories.
\begin{thmA}\label{SummThm}(Theorem~\ref{ThmFilF1})
For a fixed field $\fk$, the following statements are equivalent:
\begin{enumerate}
\item $\MdGr_{\fk}$ is filtered;
\item Every incompressible category in $\MdGr_{\fk}$ is subterminal.
\item Every incompressible category in $\MdGr_{\fk}$ is subterminal and Bezrukavnikov.
\end{enumerate}
\end{thmA}

\subsection{Maximally nilpotent and geometrically reductive
categories} 
Another interesting problem is trying to characterise {\em intrinsically} when a tensor category is incompressible. We make some progress on this question using the notions of maximal nilpotence and 
geometric reductivity introduced in \cite{ComAlg}. In particular, we 
connect it with the above questions whether 
incompressible categories must be subterminal and Bezrukavnikov. 

Namely, we say that a pre-tannakian category $\mathcal D$ is {\bf maximally nilpotent} if for any $X\in \mathcal D$, the image of the natural morphism of symmetric algebras 
$$
{\rm Sym}{\rm Ker}(X\to {\rm Hom}(X,\unit)^*\otimes \unit)\to {\rm Sym}X
$$ 
has finite length. We prove that a pretannakian category $\bT$ is  maximally nilpotent if and only if for every $X\in\bT$ 
$$\dim_{\bk}\Hom(X,\unit)\;=\;\limsup_{n\to\infty}\frac{ \log(\ell (\Sym^{\le n}X))}{\log (n)},$$
with $\ell$ the length of an object. We also show that the left-hand side is always a lower bound for the right-hand side. Using this notion and that of {\bf geometric reductivity} (for every non-zero morphism $X\tto\unit$ in $\bT$, there exists $n>0$ for which $\Sym^n X\tto\unit$ is split), we prove:

\begin{thmA} (Theorems~\ref{RemLimInc}(2) and~\ref{ThmF1})
\begin{enumerate}
\item If a finite tensor category is maximally nilpotent then it is incompressible.
\item If a pretannakian category is maximally nilpotent and geometrically reductive, then it is subterminal.
\end{enumerate}
\end{thmA}

We also establish
\begin{thmA}  (Theorems~\ref{PropVer2}, \ref{ThmVer22} and Theorem~\ref{ThmVerp})

\begin{enumerate}
\item 
The category $\Ver_{2^\infty}=\cup_{n}\Ver_{2^n}$ (equivalently, every $\Ver_{2^n}$) is geometrically reductive and maximally nilpotent, so in particular subterminal.
\item The category $\Ver_p$ is subterminal and Bezrukavnikov.
\end{enumerate} 
\end{thmA}
In particular, it follows that an affirmative answer to \cite[Question~1.2]{BE} (see \cite[Conjecture~1.4]{BEO}) would also prove that all properties in Theorem~\ref{SummThm} are satified in characteristic~$2$, and also that all incompressible categories are maximally nilpotent and geometrically reductive. This demonstrates the potential relevance of the latter two conditions. Furthermore, in \cite{ComAlg} it is observed how these two conditions are relevant for the development of algebraic geometry in these categories.

\subsection{Other results} 
Besides the above results we also derive some required results which are not directly connected to incompressible categories. In Section~\ref{TKduality} we complete Deligne's results from \cite{Del90} on generalised tannakian reconstruction, by proving a duality between affine group schemes in a pretannakian category and tensor functors to said category. In Section~\ref{SecAlgebras} we investigate for which commutative algebras in a tensor category the resulting monoidal module category is again a tensor category. Finally, in Section~\ref{SecEHO}, we use our results on $\Ver_{p^n}$ to answer some questions from \cite{EHO} about $p$-adic dimensions in tensor categories.


{\bf Acknowledgements.}  P. E.'s work was partially supported by the NSF grant DMS - 2001318. K. C.'s work was partially supported by the ARC grant DP210100251.

\section{Preliminaries}

Throughout we let $\fk$ be an algebraically closed field and set $\mN=\{0,1,2,\cdots\}.$

\subsection{Abelian categories}\label{SecCard}

\subsubsection{} For an essentially small abelian category $\cA$ we denote by $\sharp \cA$ the minimal possible cardinality of a set $E=\{X\in \cA\}$ for which every object in $\cA$ is a subquotient of a finite direct sum of $X\in E$. By additivity, either $\sharp\cA=1$ or $\sharp\cA$ is some infinite cardinality.

\begin{lemma}\label{LemCo}
Let $C$ be a coalgebra over $\fk$ and $\cC$ its category of finite dimensional comodules. For any infinite cardinal $\alpha$, we have
$$\sharp\cC <\alpha\quad\Leftrightarrow\quad \dim_{\bk}C<\alpha.$$
\end{lemma}
\begin{proof}
By taking a basis $e_i\in C$ and choosing a finite dimensional subcomodule $V_i\subset C$ containing $e_i$, for each $i$, it follows that
$$\sharp\cC\;\le\; \dim_{\bk} C.$$

Now conversely, let $\{V_i\}$ be a set (of cardinality $\sharp \cC$) of $C$-comodules such that every finite dimensional $C$-comodule is a subquotient of a finite direct sum of them. It follows that
$$\bigoplus_i V_i\otimes V_i^\ast\;\to \; C$$
is surjective. The conclusion now follows easily.
\end{proof}

\subsubsection{} We call an essentially small $\fk$-linear abelian category {\bf artinian} (over $\fk$) if it has finite dimensional morphism spaces and all objects have finite length. By Takeuchi's theorem, see~\cite[Theorem~1.9.15]{EGNO}, a category is artinian if and only if it is equivalent to the category of finite dimensional comodules of a coalgebra. 

In particular, the case $\alpha=\aleph_0$ (the cardinality of $\mN$) of Lemma~\ref{LemCo} recovers a classical result by Gabber.

\begin{corollary}\label{CorGab}(\cite[Proposition~2.14]{Del90})
An artinian category $\cC$ is the category of finite dimensional modules of a finite dimensional algebra if and only if $\sharp\cC=1$.
\end{corollary}

For a class, or category, $\cS$ of essentially small abelian categories and an infinite cardinality $\alpha$, we denote by $\cS^{<\alpha}$ the subclass, or full subcategory, of categories $\cA$ with $\sharp\cA<\alpha$.

\subsection{Tensor categories and tensor functors}

\subsubsection{}\label{12345}

Following \cite{Del90}, an essentially small $\bk$-linear symmetric monoidal category $(\bC,\otimes,\unit)$ is a {\bf tensor category over $\bk$} if
\begin{enumerate}
\item $\bC$ is abelian;
\item $\bk\to\End(\unit)$ is an isomorphism;
\item $(\bC,\otimes,\unit)$ is rigid, meaning that every object $X$ has a monoidal dual $X^\vee$.
\end{enumerate}
By \cite[Proposition~1.17]{DM}, it then follows that $\unit$ is simple. If additionally, we have
\begin{enumerate}
\item[(4)] Every object in $\bC$ has finite length,
\end{enumerate}
then using (2) and (3) shows also that morphism spaces are finite dimensional, so $\cC$ is artinian. Tensor categories satisfying (4) will be called {\bf pretannakian}. In the terminology of \cite{EGNO}, `pretannakian categories' are called `symmetric tensor categories'.

Under condition (4) the following condition is well-defined:
\begin{enumerate}
\item[(5)] For every object $X\in\bC$, the function
$$\mN\to\mN,\quad n\mapsto \ell(X^{\otimes n})$$
is bounded by an exponential.
\end{enumerate}
Pretannakian categories satisfying (5) are called (tensor categories) of {\bf moderate growth}.

A {\bf tensor subcategory} of a tensor category is said to be a topologising subcategory ({\it i.e.} a full subcategory closed under finite products (direct sums) and subquotients) closed under tensor products and duals. A tensor category is {\bf finitely generated} if it has an object $X$ which is not contained in any proper tensor subcategory, or equivalently if every object is a subquotient of a polynomial in $X,X^\vee$. Note that a finitely generated tensor category $\cC$ satisfies $\sharp \cC\le \aleph_0$, but not necessarily $\sharp\cC=1$.

We denote the category of commutative associative algebras in $\Ind\bT$ by $\Alg\cC$. We will mostly refer to them as `algebras' as we will not consider non-commutative algebras.

Following \cite{Ve} we say that a tensor category $\cC$ satisfies the {\bf Hilbert basis property} if every finitely generated algebra $A\in \Alg\cC$ is noetherian, i.e., every submodule of a finitely generated $A$-module is again finitely generated.

\subsubsection{}
We consider a chain of 2-full 2-subcategories
$$\underline{\MdGr}_{\bk}\;\subset\; \underline{\PTann}_{\bk}\;\subset\; \underline{\Tens}_{\bk},$$
where $\underline{\Tens}_{\bk}$ has as objects tensor categories over $\bk$, as 1-morphisms tensor functors ($\bk$-linear symmetric monoidal exact functors) and as 2-morphisms natural transformations of monoidal functors (which here are automatically isomorphisms). The two 2-subcategories correspond to those tensor categories which satisfy (4) and (4,5). 

Our considerations will be most natural in the 1-truncations (the categories which have as morphisms isomorphism classes of tensor functors), for which we use the same notation without underlining.

\begin{example}\label{ExInnAut}Consider two affine affine group schemes $G,H$ over $\bk$.
It is well-known that associating restriction along a homomorphism induces a bijection between $\Hom(H,G)/G(\bk)$ and $\Tens(\Rep G,\Rep H)$, where $g\in G(\bk)$ acts on $\Hom(H,G)$ via composition with the associated inner automorphism of $G$.

For example, with $S_3$ the symmetric group on 3 symbols, we have
$$|\Tens_{\bk}(\Rep_{\bk} S_3,\Rep_{\bk} \mZ/3)|=2,\quad\mbox{but}\quad |\Tens_{\bk}(\Rep_{\bk} \mZ/3,\Rep_{\bk} \mZ/3)|=3.$$
Ignoring functors which factor via $\Vecc$, these functors are induced by the embedding $\mZ/3\hookrightarrow S_3$ (unique up to conjugacy), the identity homomorphism of $\mZ/3$ and the inversion on $\mZ/3$.
\end{example}

\subsubsection{}\label{DefsVer}Assume $\mathrm{char}(\bk)=p>0$. We refer to \cite{BEO}, see also~\cite{BE,AbEnv}, for details on the chain of incompressible tensor categories of moderate growth in \eqref{VerChain}. The simple objects in $\Ver_{p^n}$ are labelled as $L_i$ for $0\le i \le p^{n-1}(p-1)-1$, with $L_0=\unit$ (we thus have a shift from the convention for $\Ver_p$ in \cite{Os}). The simple object $L_i$ in $\Ver_{p^{n-1}}$ becomes $L_{pi}$ when interpreted in $\Ver_{p^n}\supset \Ver_{p^{n-1}}$.

The category $\Ver_{p^n}$ comes with a defining monoidal functor
$$\Tilt SL_2\;\to\; \Ver_{p^n}.$$
We denote by $T_i$ the $SL_2$-tilting module with highest weight $i\in\mN$ times the fundamental weight. The above functor then satisfies $T_i\mapsto L_i$ for $0\le i\le p-1$ ($i<p-1$ if $n=1$).

The simple objects in $\Ver_{p^n}$ satisfy a `Steinberg tensor product theorem':
$$L_{p^{n-1}i_1}\otimes L_{p^{n-2}i_2}\otimes \cdots\otimes L_{p i_{n-1}}\otimes L_{i_n}\;\simeq\; L_{\sum_{j=1}^{n}p^{n-j}i_j},$$
for $0\le i_j\le p-1$ with $i_1\le p-2$, see \cite[Theorem~1.3(8)]{BEO}.

\subsubsection{Adjoints of tensor functors}\label{monad} We will use the same notation $F:\cC\to\cD$ for a tensor functor and its cocontinuous extension to the ind-completions. The latter extension always has a right adjoint, which we denote by $F_\ast$.

The functor $FF_\ast$ is then a comonad on $\Ind\cD$, with counit $\epsilon :FF_\ast\Rightarrow\Id$ and co-multiplication $F \eta F_\ast: FF_\ast\Rightarrow FF_\ast F F_\ast$. Since $F$ is faithful and exact, Barr-Beck comonadicity implies that $F$ lifts to an equivalence from $\Ind\cC$ to the category of comodules over $FF_\ast$ in $\Ind\cD$, see for instance \cite[\S 4.1]{Del90}.

Dually, $F_\ast F$ is a monad on $\Ind\cC$, yielding an equivalence between $\Ind\cD$ and the category of modules in $\Ind\cC$ whenever $F_\ast$ is faithful and exact.

\subsubsection{Deligne tensor product}\label{DTP} For two pretannakian categories $\bT_1,\bT_2$ over $\bk$, the Deligne product was introduced in \cite{Del90}. It is the coproduct of $\bT_1$ and $\bT_2$ in $\Tens_{\bk}$.

By construction, see also \cite[Proposition~6.1.3]{CEOP}, it contains the ordinary tensor product $\bT_1\otimes\bT_2$ as $\bk$-linear categories as a full subcategory such that every object in $\bT_1\boxtimes\bT_2$ is a quotient (and also a subobject) of a direct sum of objects in $\bT_1\otimes\bT_2$.

We can apply the definitions from Section~\ref{SecCard} to tensor categories:

\begin{example}
\begin{enumerate}
\item By Corollary~\ref{CorGab}, the category $\PTann_{\bk}^{<\aleph_0}$ is precisely the category of {\bf finite} (symmetric) {\bf tensor categories} in the terminology of~\cite{EGNO}.
\item We have
$$\sharp \Ver_{p^n}\,=\, 1\qquad\mbox{and}\qquad \sharp \Ver_{p^\infty}=\aleph_0.$$
\item  $\sharp(\bT_1\boxtimes\bT_2)$ is the maximum of $\sharp\bT_1$ and $\sharp\bT_2$. 
\end{enumerate}

\end{example}

\begin{lemma}\label{LemSmall0}
The categories $\PTann^{<\alpha}$ and $\MdGr^{<\alpha}$ are essentially small.
\end{lemma}
\begin{proof}
It suffices to prove this for $\PTann^{<\alpha}$, which is a direct consequence of Lemma~\ref{LemCo}.
\end{proof}

\subsection{Affine group schemes in tensor categories}

\subsubsection{} An affine group scheme $G$ in a tensor category $\bT$ is a representable functor from $\Alg\bT$ to the category of groups. It is thus represented by a commutative Hopf algebra $\cO(G)$ in $\Ind\bT$. We refer to \cite[Section~7]{ComAlg} for a detailed study of the basic theory of affine group schemes in~$\cC$ and their representations.

\subsubsection{}\label{SecFaith} Assume that the pretannakian category $\cC$ is geometrically reductive and maximally nilpotent (see Definition~\ref{Defs}), or more generally admits a tensor functor to such a category.
It is proved in \cite[Lemma~7.1.3]{ComAlg} and \cite[Theorem~6.3.1]{ComAlg} that a homomorphism of affine group schemes $H\to G$ is then a monomorphism (has trivial kernel) if and only if $H$ is a subgroup of $G$, {\it i.e.} $\cO(G)\to \cO(H)$ is an epimorphism in $\Ind\cC$.

In case $\cC$ is geometrically reductive and maximally nilpotent, or more generally admits a tensor functor to such a category, we call a representation $X$ in $\cC$ {\bf faithful} if $G\to GL_X$ is the inclusion of a subgroup. Similarly, a family of representations is called faithful if the intersections of their kernels is trivial.

\section{Potential properties of tensor categories and functors}

\subsection{Properties of tensor functors}

Let $\bT,\bT'$ be pretannakian tensor categories over $\fk$.
\begin{lemma}\label{ThmInj}
For a tensor functor $F:\bT\to\bT'$, the following are equivalent:
\begin{enumerate}
\item $F$ is fully faithful and sends simple objects to simple objects;
\item $F$ is fully faithful and every subobject of $F(X)$ is of the form $F(Y)$ for some $Y\subset X$;
\item $F$ is an equivalence between $\bT$ and a tensor subcategory of $\bT'$.
\end{enumerate}
In this case, we call $F$ {\bf injective}.
\end{lemma}
\begin{proof}
We show that (1) implies (2). We consider $Z\subset F(Y)$ and prove the claim by induction on the length of $Z$, with base case length $0$. Consider a short exact sequence $Z'\hookrightarrow Z\tto L$ for simple $L$. Then $Z'=F(Y')$ for some $Y'\subset X$, by our induction hypothesis. Moreover, by full faithfulness, we can assume that $Z'\hookrightarrow F(X)$ is induced from a $Y'\hookrightarrow X$. Hence we have a commutative diagram
$$\xymatrix{
Z'\ar@{^{(}->}[r]\ar@{=}[d]&Z\ar@{->>}[r]\ar@{^{(}->}[d]&L\ar@{^{(}->}[d]\\
F(Y')\ar@{^{(}->}[r]&F(X)\ar@{->>}[r]&F(X/Y').}$$
By considering a Jordan-H\"older filtration of $X/Y'$ and using the fact that simple objects go to simple objects, clearly $L$ must be isomorphic to $F(L')$ for a simple subquotient $L'$ of $X/Y'$. By full faithfulness we see that in fact $L'\subset X/Y'$. Let $Y$ be the pre-image of $L'$ under $X\to X/Y'$. It follows that $Z\simeq F(Y)$.

Next we show that (2) implies (3). If $F$ is fully faithful (and exact) it yields an equivalence between $\bT$ and an abelian monoidal subcategory of $\bT'$ closed under taking duals. By the extra assumption in (2), and the application of duality to it, in this case the subcategory is closed under taking subquotients and hence a tensor subcategory.

That (3) implies (1) is obvious. 
\end{proof}

\begin{remark}\label{RemInj}
In \cite{EGNO}, any fully faithful tensor functor is called injective. If $\bT$ is a finite tensor category (and $\bT'$ arbitrary), the two definitions coincide, see \cite[Proposition~6.3.1]{EGNO}. For general tensor categories, the notion of injective functor as above is stronger, consider for instance the restriction functor $\Rep G\to \Rep B$ for a Borel subgroup $B$ of a non-abelian connected reductive group $G$, e.g. $G=SL_2$ and $B$ 
the subgroup of upper triangular matrices. 

\end{remark}

\begin{example}
\label{ExDTP}
For tensor subcategories $\cD\subset\cC$ and $\cD'\subset\cC'$, the canonical tensor functor
$$\cD\boxtimes\cD'\;\to\; \cC\boxtimes\cC'$$
is injective. Indeed, one verifies easily that simple objects in $\cD\boxtimes\cD'$ are those of the form $(V_1,V_2)$ for simple objects $V_1\in\cD,V_2\in\cD'$. Full faithfulness follows from the property that every object in $\cD\boxtimes\cD'$ is a quotient as well as a subobject of some objects in $\cD\otimes\cD'$.
\end{example}

We recall \cite[Definition~1.8.4]{EGNO}.

\begin{definition}
A tensor functor $F:\bT\to\bT'$ is {\bf surjective} if every object in $\bT'$ is a subquotient of one in the essential image of $F$.
\end{definition}

\begin{example}\label{ExGH} \cite{DM}
Consider a homomorphism $f:H\to G$ of affine group schemes over~$\fk$, with corresponding restriction functor $F: \Rep G\to \Rep H$.
\begin{enumerate}
\item The functor $F$ is injective if and only if $f$ is faithfully flat (which here just means that $\cO(G)\to\cO(H)$ is injective).
\item The functor $F$ is surjective if and only if $f$ is a closed immersion. 
\end{enumerate}
\end{example}

\begin{corollary}\label{CorEq}
\begin{enumerate}
\item Every tensor functor is a composition of a surjective followed by an injective tensor functor, in a unique way (up to equivalence).
\item A tensor functor is an equivalence if and only if it is both surjective and injective.
\end{enumerate}
\end{corollary}
\begin{proof}
Part (1) is a triviality. Part (2) follows from~\autoref{ThmInj}, or from \cite[Proposition~4.4.10]{CEH}.
\end{proof}

The following observation is straightforward, but useful.

\begin{remark}\label{RemSq}
Consider the commutative square (of solid arrows) in $\PTann_{\fk}$
$$\xymatrix{
\bT_1\ar@{->>}[rr]\ar[d]&& \bT_2\ar[d]\ar@{-->}[dll]\\
\bT_3\ar@{^{(}->}[rr]&&\bT_4
}$$
where the upper horizontal arrow is surjective and the lower horizontal arrow is injective. Then $\bT_2\to\bT_4$ actually factors through $\bT_3\to \bT_4$.
\end{remark}

\subsection{Potential properties of tensor categories}

Part (1) of the following definition was introduced by the third author, see \cite[Definition~4.3]{BE}. Note that we actually use a slight variation of the latter definition of incompressibility, which could in principle deviate for non-finite incompressible categories (see Remark~\ref{RemInj}). However, all known examples of incompressible categories are (unions of) finite tensor categories. Parts (4) and (5) were introduced in \cite{ComAlg}, the latter as a direct generalisation of the same term for algebraic groups.

\begin{definition}\label{Defs}
Let $\cD\in\PTann_{\bk}$.
\begin{enumerate}

\item $\cD$ is {\bf incompressible} if one of the following equivalent conditions is satisfied:
\begin{enumerate}
\item Every tensor functor $\cD\to\bT$ is injective;
\item Every surjective tensor functor $\cD\to\cC$ is an equivalence.
\end{enumerate} 

\item $\cD$ is {\bf subterminal} if for every tensor category $\bT$, there is at most one (up to isomorphism) tensor functor $\bT\to\bD$ ({\it i.e.} $\cD$ is subterminal in the category $\PTann_{\bk}$).
\item $\cD$ is a {\bf Bezrukavnikov category} if for each pair of a surjective tensor functor $\bT_1\to \bT_2$ between pretannakian categories and a tensor functor $\bT_1\to\bD$, there also exists a tensor functor $\bT_2\to \bD$.
\item $\cD$ is {\bf Geometrically Reductive} (or simply {\bf GR}) if for every morphism $X\tto\unit$ in $\cC$, there exists $n>0$ for which $\Sym^n X\tto\unit$ is split.
\item $\cD$ is {\bf Maximally Nilpotent} (or simply {\bf MN}) if:
\begin{enumerate}
\item For every non-trivial simple $L\in\bT$, the algebra $\Sym(L)$ is finite, and
\item For every non-split $\unit\to X$ in $\bT$, $\unit\to\Sym^n X$ is zero for some $n\in\mN$.
\end{enumerate}
\end{enumerate}

\end{definition}

In order to motivate briefly the new definitions, we can observe that a necessary condition for $\MdGr_{\bk}$ to have a terminal object is that every incompressible category in $\MdGr_{\bk}$ is subterminal and a Bezrukavnikov category.

\begin{question}\label{QIncomp}
Is every tensor subcategory of an incompressible pretannakian category incompressible?
\end{question}
We will answer this question affirmatively for finite tensor categories in Section~\ref{SecAlgebras}.

\begin{remark}\label{RemLimInc}
\begin{enumerate}
\item That the two characterisations in \ref{Defs}(1) are indeed equivalent follows immediately from Corollary~\ref{CorEq}.
\item If $\bT$ is a finite tensor category, it is incompressible if and only if every tensor functor $\bT\to\bT'$ is full, see Remark~\ref{RemInj}.
\item A direct limit of incompressible categories is incompressible.
\end{enumerate}
\end{remark}

\begin{example}
In \cite{B}, Bezrukavnikov proved that $\Vecc$ satisfies the condition in Definition~\ref{Defs}(2). The same proof applies to $\sVec$. An alternative proof follows from \cite[Lemma~3.3.2(ii) and Proposition~3.3.5(iii)]{Tann}.
\end{example}

\subsection{Symmetric algebras}
Let $\cC$ be a tensor category over $\bk$.
\subsubsection{}
We consider a short exact sequence
$$0\to X\to Y\to Z\to 0$$
in $\cC$. This filtration on $Y$, with $\gr Y\simeq X\oplus Z$, induces a filtration on the tensor and symmetric powers of $Y$ and we have a canonical (bigraded) epimorphism in $\Ind\cC$
\begin{equation}\label{grSym}
\Sym(X)\otimes \Sym(Z)\simeq \Sym(\gr Y)\;\tto\; \gr\Sym(Y).
\end{equation}
This epimorphism is always an isomorphism if and only if $\cC$ is Frobenius exact, see~\cite{Tann, EOf}. In this section we show that in general we can easily control the kernel of ~\eqref{grSym}.

\begin{theorem}\label{ThmSym}
Let $R\in\Alg\cC$ denote the image of the algebra morphism $\Sym(X)\to\Sym(Y)$. Then \eqref{grSym} induces an isomorphism
$$R\otimes \Sym (Z)\;\simeq\;\gr\Sym( Y).$$
\end{theorem}

We need some preparation for the proof. By a graded (co)algebra 
in $\cC$ we will mean a $\mN$-graded (co)algebra $A$ in $\Ind\cC$ with graded pieces in $\cC$
and $A[0]=\unit$. Let $A^\vee=\oplus_i A[i]^\vee$ be the graded dual to a graded algebra (coalgebra) $A$. 
Then $A^\vee$ is a graded coalgebra (algebra).

\begin{lemma}\label{l1}
Let $A,B$ be graded coalgebras in $\cC$. 
For a bigraded coideal $I\subset A\otimes B$, we have 
$$I=(I\cap A)\otimes B + A\otimes (I\cap B).$$   
\end{lemma}  
 
\begin{proof} 
We can replace $A$ by $A/(I\cap A)$ and $B$ by $B/(I\cap B)$, so that it is sufficient to show that $I\cap A=0=I\cap B$ implies that we have $I=0$.
Let $C=I^\perp \subset A^\vee\otimes B^\vee$, a bigraded subalgebra. 
Then $C$ contains $A^\vee$ and $B^\vee$, hence $C=A^\vee\otimes B^\vee$, or equivalently $I=0$. 
\end{proof}

\begin{proof}[Proof of Theorem~\ref{ThmSym}]
We consider the canonical Hopf algebra structure of $\Sym( Y)$ (with $Y$ primitive). The coproduct respects the filtration on $\Sym (Y)$, so that \eqref{grSym} becomes a morphism of Hopf algebras. We can thus apply Lemma~\ref{l1} and the fact that $\Sym (Z)\to \gr\Sym (Y)$ is a monomorphism (since we can compose with the epimorphism $\gr \Sym (Y)\tto \Sym (Z)$ (coming from $\Sym (Y)\tto \Sym (Z)$) to get the identity on $\Sym (Z)$.
\end{proof}


\section{Tannaka - Krein duality for tensor categories}\label{TKduality}

Recall that $\bk$ is an algebraically closed field and fix a pretannakian category $\bD$ over~$\bk$. All results in this section remain valid for perfect fields, or even arbitrary fields if we assume that the base category $\bD$ is `absolute', see~\cite{CEOP}.

\subsection{$\cD$-groups}

\subsubsection{} 
To any tensor functor $F:\bT\to\bD$, we can associate an affine group scheme $\uAut(F)$ in~$\bD$ as in \cite[\S 8]{Del90}, with
$$\cO(\uAut(F))\;=\;\int^{X\in\bT} (FX)\otimes (FX)^\vee.$$ In particular, we have the `fundamental group' $\pi(\bD)=\uAut(\Id_{\bD})$.
By construction, the affine group scheme $\pi(\bD)$ acts {\em canonically} on every object of $\Ind\bD$. We refer to \cite[\S 8]{Del90} for details.

\subsubsection{}
For an affine group scheme $G$ in $\bD$ equipped with a homomorphism 
$$
\phi: \pi(\bD)\to G
$$ 
consider the adjoint coaction
$$\cO(G)\to \cO(G)^{\otimes 3}\to \cO(\pi(\bD))\otimes \cO(G)\otimes \cO(\pi(\bD))\to \cO(\pi(\bD))\otimes \cO(G),$$
where the first arrow is twice the comultiplication, the second arrow comes from the morphism $\phi^*: \cO(G)\to \cO(\pi(\bD))$ in the left and right factor, followed by the antipode on the right factor, and the third arrow is multiplication of the two factors $\cO(\pi(\bD))$.

If this morphism corresponds to the canonical coaction on the object $\cO(G)\in \Ind\bD$, we call the diagram $\pi(\bD)\to G$ a {\bf $\bD$-group}. The category of $\bD$-groups is the corresponding full subcategory of the slice category of affine group schemes under $\pi(\bD)$. In other words, a morphism $(H,\psi)\to (G,\phi)$ is a homomorphism $f:H\to G$ such that $f\circ \psi=\phi$.

\begin{example}\label{ExTriv}
For an affine group scheme $G$ over~$\bk$ interpreted as an affine group scheme in $\bD=\sVec$ (where $\pi(\sVec)=\mZ/2$), a diagram $\phi:\mZ/2\to G$ is a $\sVec$-group if and only if $\phi$ is central.
\end{example}

\subsubsection{}For a tensor functor $F:\bT\to\bD$, the obvious morphism $\psi:\pi(\bD)\to\uAut(F)$ yields a $\bD$-group. Conversely, to any $\bD$-group $\phi:\pi(\bD)\to G$, we can associate the tensor category $\Rep(G,\phi)$ of $G$-representations in $\bD$ such that the restriction to $\pi(\bD)$ yields the canonical $\pi(\bD)$-action. We denote the forgetful functor to $\bD$ by $\omega$.  If in Example~\ref{ExTriv} we let $\phi$ be trivial, then $\Rep(G,\phi)$ yields the ordinary category of (non-super) representations of $G$.

The category of all $G$-representations in $\bD$ is denoted by $\Rep G$. By \cite[Lemma~7.3.2]{ComAlg} the ind-completion of $\Rep G$ is the category of all $\cO(G)$-comodules in $\Ind\bD$.

\subsubsection{} Consider the following 2-category. Objects are tensor functors $F:\cC\to\cD$, for arbitrary $\cC\in\Tens_{\bk}$ and our fixed $\cD\in\PTann_{\bk}$. The 1-morphisms from $F_1:\cC_1\to\cD$ to $F_2:\cC_2\to\cD$ are given by pairs $(G,\alpha)$ of a tensor functor $G:\cC_1\to \cC_2$ and a natural transformation (automatically an isomorphism), $\alpha: F_1\Rightarrow F_2\circ G$. Finally, a 2-morphism from $(G,\alpha)$ to $(G',\alpha')$, where both are 1-morphisms from $F_1$ to $F_2$ as above, is given by a natural transformation $\beta: G\Rightarrow G'$ which satisfies $(F_2\beta)\circ \alpha =\alpha'$.

Since $F_2$ is faithful and $\alpha$ invertible, such a $\beta$ is necessarily unique, so the 2-category is actually equivalent to its 1-truncation, which we denote by $\PTann(\cD)$. Note that isomorphisms in $\PTann(\cD)$ correspond to (classes of) pairs $(G,\alpha)$ where $G$ is an equivalence.

\begin{remark}
Recall the category $\cD\mbox{-}\Tann$ from Section~\ref{SecDTann} and consider the slice category $\PTann_{\bk}/\cD$ of $\cD$ as an object in the category $\PTann_{\bk}$. Morphisms in the latter are thus equivalence classes of tensor functors leading to commutative triangles. We have obvious forgetful functors
$$\PTann(\cD)\;\to\; \PTann_{\bk}/\cD\;\to\; \DTann.$$
These are generally not equivalences.
Obviously the second functor is an equivalence if and only if $\cD$ is subterminal. However, even when $\cD$ is subterminal, the first functor is not faithful (but always full), see for instance Example~\ref{ExInnAut} and the (classical well-known) case $\cD=\Vecc$ in Theorem~\ref{ThmTann}.
\end{remark}

\subsection{Tannaka-Krein duality}

\begin{theorem}\label{ThmTann}
The assignment
$$(F: \bT\to\bD)\quad \mapsto\quad \left( \psi:\pi(\bD)\to \underline{\Aut}^{\otimes}(F)\right)$$
yields a contravariant equivalence between $\PTann(\cD)$ and the category of $\bD$-groups, with inverse given by
$$(\phi:\pi(\bD)\to G)\quad\mapsto\quad (\omega:\Rep(G,\phi)\to \bD).$$
\end{theorem}
\begin{proof}
We start by verifying that the first assignment corresponds to a functor. We consider a morphism in $\PTann(\cD)$, which is represented by a pair $(G,\alpha)$:
$$
\xymatrix{
\cC_1\ar[rd]^G\ar[dd]_{F_1}\\
\ar@{=>}[r]^-{\alpha}&\cC_2\ar[ld]^{F_2}\\
\cD.
}$$
This morphism is sent to the homomorphism of $\cD$-groups
$$\underline{\Aut}^{\otimes}(F_2)\to \underline{\Aut}^{\otimes}(F_1),\quad\xi\mapsto \alpha^{-1}\circ (\xi G)\circ \alpha.$$
It then follows directly that this does not depend on the representative $(G,\alpha)$.

It is straightforward to see that the second assignment corresponds to well-defined functors. By \cite[Th\'eor\`eme~8.17]{Del90}, the canonical tensor functor (which lifts $F$)
$$\bT\;\to\; \Rep(\underline{\Aut}^{\otimes}(F),\psi)$$
is an equivalence.
It follows that the composition of the functors in order of their appearance in the theorem is isomorphic to the identity.

For the other composition, we start from a $\bD$-group $(G,\phi)$. We set $H:= \underline{\Aut}^\otimes(\omega)$ and consider the natural homomorphism $\psi:\pi(\bD)\to H$. We have a canonical homomorphism of $\bD$-groups $(G,\phi)\to (H,\psi)$ coming from the co-action $$\omega X \otimes\omega X^\vee\to \cO(G).$$ In particular it follows that the composite of the equivalence from the first paragraph with restriction along $G\to H$,
$$\Rep(G,\phi)\;\xrightarrow{\sim}\; \Rep(H,\psi)\;\to\; \Rep(G,\phi),$$ 
is isomorphic to the identity. In particular, the second arrow is also an equivalence, 
so by Proposition~\ref{PropTann} below $G\to H$ is an isomorphism. The conclusion thus follows.
\end{proof}

\begin{lemma}\label{LemTan1}
If a homomorphism $G\to H$ of affine group schemes in $\bD$ induces an equivalence
$$\Rep H\to\Rep G,$$
then it must be an isomorphism.
\end{lemma}
\begin{proof}
This follows from the observation that with $\Omega:\Rep G\to\bD$ the forgetful functor, we have an isomorphism of comonads
$$\Omega\Omega_\ast\;\simeq\; \cO(G)\otimes -,$$
showing that $\cO(H)\to\cO(G)$ must be an isomorphism.
\end{proof}

For a $\bD$-group $(G,\phi)$, we have obvious tensor subcategries
$$\Rep(G,\phi)\;\subset\; \Rep G\;\supset\;\bD,$$
leading to the functor in the lemma below.
This lemma is the crucial technical result in our proof of Theorem~\ref{ThmTann}. We point out that it can also be considered a consequence of the theorem, via \cite[Proposition~8.22]{Del90}.
\begin{lemma}\label{LemTan2}
For a $\bD$-group $(G,\phi)$, the canonical tensor functor
$$F:\Rep(G,\phi)\boxtimes\bD\;\to\;\Rep G $$
is an equivalence.
\end{lemma}
\begin{proof}
We consider the commutative diagram of tensor functors
$$
\xymatrix{
\Rep(G,\phi)\boxtimes\bD\ar[rr]^F\ar[d]^\Psi&&\Rep G\ar[d]^\Phi\\
\bD\boxtimes\bD\ar[rr]^{\sim}_H&&\Rep \pi(\bD).
}
$$
Here $\Psi=\omega\boxtimes \Id_{\bD}$ (the tensor functor defined from $\omega$ and $\Id_{\bD}$ via the coproduct property), while $\Phi$ is restriction along $\phi$. Finally, $H$ is the equivalence of tensor categories from \cite[Corollaire 8.23]{Del90}, which sends $(X,Y)\in\cD\otimes\cD$ to $X\otimes Y$ in $\Rep\pi(\cD)$ where the $\pi(\cD)$ action on $X\otimes Y$ comes from the canonical action on $X$ and trivial action on $Y$.

We claim that the canonical natural transformation $ F\Psi_\ast\Rightarrow \Phi_\ast H$ is an isomorphism. Indeed, both compositions are left exact, so it suffices to prove the claim on the full subcategory $\bD\otimes\bD$ of $\bD\boxtimes\bD$, see \ref{DTP}. The latter is now a straightforward exercise using the following description of $\omega_\ast$; it is given by sending an object $X\in\bD$ to the $\pi(\bD)$-invariants in $\cO(G)\otimes X$, for the action coming from the right action on $\cO(G)$ via $\phi$ and the canonical left action on $X$. Indeed, this is equivalent to taking the maximal $G$-subrepresentation of $\cO(G)\otimes X$ for which the restriction of the regular left $G$-action to $\pi(\bD)$ equals the canonical $\pi(\bD)$-action.

The claim from the previous paragraph states that the equivalence $H$ (or more precisely its extension to the ind-completions) exchanges the comonads $\Psi\Psi_\ast$ and $\Phi\Phi_\ast$, proving that~$F$ is an equivalence via Barr-Beck comonadicity, see \ref{monad}.
\end{proof}

\begin{prop}\label{PropTann}
If a homomorphism $(G,\phi)\to (H,\psi)$ of $\bD$-groups induces an equivalence
$$\Rep (H,\psi)\to\Rep (G,\phi),$$
then it must be an isomorphism.
\end{prop}
\begin{proof}
This follows from the combination of Lemmata~\ref{LemTan1} and~\ref{LemTan2}.
\end{proof}


\begin{lemma}\label{LemTan3}  Assume that $\cD$ is {\bf GR} and {\bf MN} (or admits a tensor functor to such a category).
For a $\cD$-group $(G,\phi)$, consider a tensor subcategory $\cC\subset\Rep(G,\phi)$ and the canonical tensor functor
$$F_{\cC}:\cC\boxtimes\bD\;\to\;\Rep G$$
\begin{enumerate}
\item
The following properties are equivalent:
\begin{enumerate}
\item $\cC$ contains a faithful family of $G$-representations;
\item $F_{\cC}$ is surjective;
\end{enumerate} 
\item The properties in (1) are satisfied if and only if $\cC=\Rep(G,\phi)$.
\end{enumerate}
\end{lemma}
\begin{proof}

By \ref{SecFaith}, we have an unambiguous notion of faithful (collections of) representations.

We start with part (1). Consider a faithful family $\{X_\alpha\in \Rep(G,\phi)\}$. By definition, we have an epimorphism
$$\bigotimes_\alpha \cO(GL(X_\alpha))\;\tto\; \cO(G)$$
in $\Ind\cD$, which can be interpreted in $\Ind\Rep G$, for the canonical left actions. Since $\cO(GL(X_\alpha))$ (as a $G$-representation) can be written as a quotient of the algebra
$$\Sym(X_\alpha\otimes\omega X_\alpha^\vee)\otimes\Sym(\omega X_\alpha\otimes X_\alpha^\vee),$$
it follows that the family
$$\{(X_\alpha,\omega X_\alpha^\vee)\}\cup\{(\unit,X)\,|\, X\in\cD\}$$
of objects in $\Rep(G,\phi)\otimes\cD\subset \Rep G$
generates $\Rep G$. Conversely, it is clear that if a collection of objects $\{(X_\alpha,Y_\beta)\in \Rep(G,\phi)\otimes \cD\}$ generates $\Rep G$, then $\{X_\alpha\}$ is a faithful family.

For part (2), if $\cC=\Rep(G,\phi)$, then $F_{\cC}$ is surjective by Lemma~\ref{LemTan2}. Conversely, assume that $F_\cC$ is surjective, and therefore an equivalence by Lemma~\ref{LemTan2} and Example~\ref{ExDTP}. We can identify $\cC=\Rep(H,\chi)$ for $H$ the automorphism group of the restriction of $\omega$ to $\cC$, and have a canonical morphism $G\to H$ which induces the inclusion. It is an easy consequence of Lemmata~\ref{LemTan1} and~\ref{LemTan2} that $(G,\phi)=(H,\chi)$, so $\cC=\Rep(G,\phi)$.\end{proof}


\begin{remark}
An alternative proof of the crucial Lemma~\ref{LemTan2} would be showing directly that $F$ is injective and surjective. 

One verifies easily that $F$ is fully faithful, so that injectivity corresponds to proving that $L_1\otimes L_2$ is simple in $\Rep G$ for simple objects $L_1\in\Rep(G,\phi)$ and $L_2\in\cD$.

By Lemma~\ref{LemTan3}(1), which does not depend on Lemma~\ref{LemTan2}, for surjectivity we need to prove that $F$ contains a faithful collection of representations. The latter is clear when~$\cD$ is semisimple. Concretely, we can take the collection of (finite subrepresentations of) $\{\omega_\ast(X)\}$, with $ X\in\bD$. Indeed, in $\Ind\Rep G$ we have
$$\bigoplus_{X\in\cC}\omega_\ast(X)\otimes X^\vee \simeq\Phi_\ast (\cO(\pi(\bD)))\simeq \cO(G),$$
for the left regular $G$-action on $\cO(G)$.
\end{remark}


%
%

\subsection{Tensor subcategories}

Assume for the entire section that $\cD$ is {\bf GR} and {\bf MN}.

\begin{theorem}\label{ThmNSub}

For a $\bD$-group $(G,\phi)$, the assignment
$$(N\lhd G)\mapsto \Rep(G/N,\phi)$$
yields a bijection between kernels of homomorphisms out of $G$ and tensor subcategories of $\Rep(G,\phi)$.
\end{theorem}
\begin{proof}
For a tensor subcategory $\bC\subset \Rep(G,\phi)$, we consider the intersection $N\lhd G$ of all kernels of representations contained in $\bC$.
That the two assignments are mutually inverse can be deduced from the special case that they exchange the trivial subgroup of $G$ and~$\Rep(G,\phi)$ regarded as a subcategory of itself. The latter is just a reformulation of Lemma~\ref{LemTan3}(2).
\end{proof}
%
%
%

\begin{remark}
Excluding the case $\bD=\Vecc$, the statement in Theorem~\ref{ThmNSub} is not true about tensor subcategories of $\Rep G$.
\end{remark}

\begin{lemma}\label{LemAlgGrp}
For a $\bD$-group $(G,\phi)$, the affine group scheme $G$ is of finite type ($\cO(G)$ is a finitely generated algebra) if and only $\Rep (G,\phi)$ is finitely generated.
\end{lemma}
\begin{proof}
In \cite[Lemma~7.3.3]{ComAlg}, it was proved that $G$ is of finite type if and only if $\Rep G$ is generated by $\bD$ plus one additional object. If $\Rep (G,\phi)$ is generated by $X$, then it follows from Lemma~\ref{LemTan2} that $\Rep G$ is generated by $(X,\unit)\in \Rep(G,\phi)\otimes \bD$ (under the equivalence in the lemma) and~$\bD$. Conversely, if there is a single representation $X$ in $\Rep G$ which generates the category together with $\bD$, then the representation must be faithful. Viewing $X$ as a quotient of some object $(X_1,Y_1)\in\Rep(G,\phi)\otimes \bD$ demonstrates that $X_1$ must be a faithful representation and therefore generates $\Rep(G,\phi)$, by Lemma~\ref{LemTan3}(2).
\end{proof}

\begin{corollary}\label{CorAlgGrp}
 If a tensor category $\bT$ which admits a tensor functor $\bT\to\bD$ is finitely generated, then every tensor subcategory of $\bT$ is also finitely generated.
\end{corollary}
\begin{proof}We can write $\bT=\Rep(G,\phi)$ for a $\bD$-group $(G,\phi)$.
Consider an increasing chain 
$$\bT_1\subset\bT_2\subset\cdots\subset \bT$$
of tensor subcategories of $\bT$. By Theorem~\ref{ThmNSub} we get a corresponding decreasing chain of normal subgroups $N_i\lhd G$, leading to an increasing chain of ideals $I_i$ in $\cO(G)$. By Lemma~\ref{LemAlgGrp} $\cO(G)$ is finitely generated. Since $\cD$ satisfies the Hilbert basis property (see \cite[Lemma~6.4.4]{ComAlg}) $\cO(G)$ is noetherian, showing that the ideals, and therefore the chain of tensor subcategories, must stabilise.
\end{proof}
\begin{remark}
We know of no examples of finitely generated pretannakian categories with a tensor subcategory which is not finitely generated.
In contrast, when considering tensor categories without braiding, there are clear examples. For instance, consider the non-symmetric tensor category $\Vecc_\Gamma$ of finite dimensional 
$\Gamma$-graded vector spaces, for a free group $\Gamma$ in $n\ge 2$ generators. Then $\Vecc_\Gamma$ is finitely generated but $\Vecc_{\Gamma'}\subset \Vecc_{\Gamma}$ is not for infinitely generated (free) subgroups $\Gamma'\subset\Gamma$.
\end{remark}


\section{Structure of the category of tensor categories}

\subsection{Cocones and colimits}

\begin{prop}\label{PropFilt}
Assume $\mathrm{char}(\fk)=0$. The categories $\PTann_{\fk}$ and $\Tens_{\bk}$ are not filtered.
\end{prop}
\begin{proof}
We will prove that in both categories there exist parallel arrows which cannot be equalised.

We start with $\PTann_{\bk}$. Let $(\Rep GL)_0^{ab}$ be the universal tensor category on an object $X_0$ of dimension 0 which is not annihilated by any Schur functor, see \cite{EHS, Del07, AbEnv}. We can consider two endo-tensor functors of $(\Rep GL)_0^{ab}$: the identity functor and the one which sends $X_0$ to a direct sum of two copies of itself. Any tensor functor $F$
$$(\Rep GL)_0^{ab}\rightrightarrows (\Rep GL)_0^{ab}\xrightarrow{F}\bT$$
which would equalise the two tensor functors on the left would satisfy $$
F(X_0)\simeq F(X_0)\oplus F(X_0),$$ which is not possible in an artinian (and hence Krull-Schmidt) category $\bT$.

A similar example, which avoids using the abelian envelopes of Deligne's universal categories, but uses Deligne's tensor product, can be obtained from tensor functors 
$$
(\Rep GL)_t\rightrightarrows (\Rep GL)_t\boxtimes\sVec,
$$ 
for $t\in\fk\backslash\mZ$.

Now we consider $\Tens_{\bk}$ (and give an example which also implies the lemma for $\PTann_{\bk}$). We use $(\Rep S)_t$ from \cite{Del07}. For $t\in\bk\backslash \mZ$, consider the tensor functors
$$F_1,F_2: (\Rep GL)_t\;\rightrightarrows\; (\Rep S)_t\boxtimes \sVec,$$
with 
$$F_1(X_t)= A:= (T_t, \unit),\;\mbox{ and }\;\; F_2(X_t)=B:= (T_t, \unit) \oplus (\unit, \bk^{1|1}).$$
We will prove at the end of Section~\ref{SecAlgebras} that these tensor functors cannot be equalised.
\end{proof}

\begin{remark}
If $\fk$ is not algebraically closed, then even $\MdGr_{\fk}$ is not filtered. Indeed, in this case, tensor functors to $\Vecc_{\fk}$ need not be unique, see \cite[Theorem~3.2]{DM}; and there is no possibility of equalising them with a tensor functor out of $\Vecc_{\fk}$ (without extending scalars).

%
\end{remark}

\begin{question}
In the proof of Proposition~\ref{PropFilt} we use the observation that pretannakian categories cannot contain `Banach-Tarski objects', {\it i.e.} objects $X$ satisfying $X\simeq X\oplus X$.
Can tensor categories contain Banach-Tarski objects? Note that this is not allowed in tensor categories of the form $\varinjlim \cC_i$ (for a direct system $\{\cC_i\}$ of pretannakian categories), see for instance~\cite[\S 2.19]{Del90}, or $\prod_{\cU}\cC_n$ (for an ultrafilter~$\cU$ on a labelling set of a family of pretannakian categories), see~\cite{Ha}, which are the currently known examples of tensor categories that are not necessarily pretannakian. 

Such objects do exist if we remove restriction \ref{12345}(3) of rigidity from our definition of tensor categories. Indeed, one can then consider the abelian monoidal category (with exact tensor product) of countably dimensional vector spaces. Furthermore, if we do impose rigidity but omit conditions \ref{12345}(1) and (2), one can easily construct (universal) examples using the techniques from \cite{IKO}. With further work, one can even construct an example which satisfies all conditions except the abelian condition \ref{12345}(1).

\end{question}

Now we reformulate some results from \cite{CEO, Del90}.

\begin{prop}\label{PropFiltLim}
\begin{enumerate}
\item The categories $\MdGr_{\fk} $ and $\Tens_{\fk} $ are closed under filtered colimits, moreover filtered colimits in $\MdGr_{\fk} $ can be calculated in $\Tens_{\fk} $.
\item The category $\PTann_{\fk} $ is not closed under filtered colimits.
\end{enumerate}
\end{prop}
\begin{proof}
The naive filtered colimit of a system of tensor categories is canonically again a tensor category, see for instance~\cite[\S 6]{CEO}. In~\cite[Lemma~6.3]{CEO}, it is proved that such filtered colimits of categories $\MdGr_{\fk} $ remain in $\MdGr_{\fk} $.

For part (2) it suffices to construct a chain
$$\bT_1\to\bT_2\to\bT_3\to\cdots $$
in $\PTann_{\fk}$ for which there exists $X=X_1\in\bT_1$ such that the function $i\mapsto \dim_{\fk}\End(X_i)$ is strictly increasing, with $X_i$ the image of $X$ in $\bT_i$, since such a diagram cannot have a cone in $\PTann_{\fk}$.
An example for $\mathrm{char}(\fk)=0$ ($\cC_i=(\Rep GL)_{t-i}$) is given in \cite[2.19]{Del90}.
Similar examples in positive characteristic can be derived from~\cite[Section~3.3]{Ha}.
\end{proof}


\subsection{Existence of tensor functors to incompressible categories}

In this section we prove the following stronger (due to Question~\ref{QIncomp}) version of Theorem~\ref{ThmIncom}. 
\begin{theorem}\label{ThmVic}
Every pretannakian category of moderate growth $\cC$ admits a surjective tensor functor to an incompressible category (of moderate growth).
\end{theorem}

\subsubsection{} We begin preparations for the proof of Theorem \ref{ThmVic}. Fix a pretannakian category $\bT$. Consider the 2-category $\underline{\mathcal{S}}=\underline{\mathcal{S}}(\bT)$ where objects are surjective tensor functors $F:\bT\to\bT_1$ and 1-morphisms from $F_1:\bT\to\bT_1$ to $F_2:\bT\to\bT_2$ are pairs $(G,\alpha)$ of a (necessarily surjective) tensor functor $G:\bT_1\to\bT_2$ and a natural transformation of tensor functors (automatically isomorphism) $\alpha:G\circ F_1\Rightarrow F_2$. Finally, a 2-morphism from $(G,\alpha)$ to $(G',\alpha')$ is a natural transformation $\beta:G\Rightarrow G'$ with $\alpha'\circ(\beta F_1)=\alpha$.

Clearly $\underline{\cS}$ is equivalent to its 1-truncation, which we denote by $\cS=\cS(\bT)$, where morphisms correspond to isomorphism classes of tensor functors.


\begin{lemma}\label{LemSmall}\label{LemLim}
\begin{enumerate}
\item The category $\cS(\bT)$ is essentially small.
\item 
If $\bT$ is of moderate growth, $\cS(\bT)$ is closed under filtered colimits.
\end{enumerate} 
\end{lemma}
\begin{proof}
Part (1) follows from Lemma~\ref{LemSmall0}.
It follows from \cite[Lemma~4.9]{CEO} that every target of a surjective tensor functor out of $\bT$ has again moderate growth. Part (2) thus follows from Proposition~\ref{PropFiltLim}(1).
\end{proof}

%
%

\begin{lemma}\label{LemVictor}
Let $\cA$ be an artinian category with $Y\in \cA$ such that every object in $\cA$ is a subquotient of $Y^n$ for some $n\in\mN$ ({\it i.e.} $\sharp\cA=1$). Then every exact faithful $\fk$-linear endofunctor~$G$ of~$\cA$ which satisfies $GY\simeq Y$ is an auto-equivalence.
\end{lemma}
\begin{proof}
Firstly, we observe that $G$ preserves the length of objects. By faithfulness and exactness, it can only increase length. The fact that $GY\simeq Y$ and the generating property of $Y$ then show that $G$ cannot strictly increase length.

By Corollary~\ref{CorGab}, the category $\cA$ is equivalent to the category of finite dimensional $A$-modules for a finite dimensional associative $\fk$-algebra $A$. Without loss of generality we let $A$ be basic, so that its simple modules are 1-dimensional.

Since $G$ is $\fk$-linear and exact, it is of the form $M\otimes_A-$ for an $A$-bimodule $M$ which is flat (so projective) as a right $A$-module. Furthermore, faithfulness then implies that every indecomposable projective must appear as a direct summand.

Now the condition that $G$ preserves length can be translated to the fact that it preserves dimension of $A$-modules. In particular $M\simeq M\otimes_A A$ must have the same dimension as $A$. In combination with the conclusion from the previous paragraph, we conclude that $M\simeq A$ as a right $A$-module. The left $A$-module structure on $A$ (which must commute with the regular right action) must correspond to an algebra endomorphism $\varphi :A\to A$, and hence~$M$ is of the form~$_{\varphi}A$, or equivalently $G$ can be interpreted as the functor which twists the $A$-action by~$\varphi$.

The $A$-module $Q$ corresponding to $Y$ is clearly faithful and so also $_{\varphi}Q\simeq Q$ must be faithful. Since the action of $A$ on $_{\varphi}Q$ factors as
$$A\xrightarrow{\varphi}A\to \End_{\bk}(Q),$$
this shows that $\varphi$ is injective and hence an automorphism. It follows that $G$ is an auto-equivalence.
\end{proof}

\begin{lemma}\label{LemHF}
Consider a surjective tensor functor $F:\bT\to \bT_1$. A tensor functor $H:\bT_1\to\bT_1$ with $H\circ F\simeq F$ must be an auto-equivalence.
\end{lemma}
\begin{proof}
For each $X\in\bT$, denote by $\langle F(X)\rangle\subset \bT_1$ the full subcategory of subquotients of direct summands of $F(X)$.
It follows from Lemma~\ref{LemVictor} that $H$ restricts to an auto-equivalence on $\langle F(X)\rangle$. By assumption, $\bT_1$ is the union of such subcategories, so $H$ is an equivalence.
\end{proof}

%
%

The following lemma is a slight generalisation of Zorn's lemma (and recovers it for $\cC$ the category associated to a poset).

\begin{lemma}\label{LemZorn}
Let $\cC$ be an essentially small category such that every directed system has a cocone and such that for morphisms
$$X\xrightarrow{f}Y\xrightarrow{g}Y$$
the condition $f=g\circ f$ implies that $g$ is an isomorphism. Then $\cC$ contains an object $X$ such that every morphism $X\to ?$ in $\cC$ is an isomorphism.
\end{lemma}
\begin{proof}
Firstly, consider a retract 
$$B\xrightarrow{i} A\xrightarrow{r}B$$
in $\cC$, {\it i.e.} $r\circ i=\id_B$. Since $i\circ r\circ i=i$, we find that $i\circ r$ is an isomorphism, and subsequently both $i,r$ are isomorphisms.
Moreover, since the conclusion is preserved under equivalence, we can assume that $\cC$ is small.

Now assume for a contradiction that no such $X$ exists.
As in the proof of Zorn's lemma, we can construct by transfinite induction and taking cocones a chain
$$X_1\to X_2\to X_3 \to\cdots \to X_{\omega}\to X_{\omega+1}\to \cdots$$
labelled by all ordinals, where no $X_\alpha\to X_{\alpha+1}$ is an isomorphism. For every limit ordinal $\alpha$, the arrow $\to X_\alpha$ symbolises the cocone.  That no morphism $X_\alpha\to X_\beta$ is an isomorphism then follows from considering the diagram $X_\alpha\to X_{\alpha+1}\to X_\beta$. We just saw the first morphism is not an isomorphism, so the composite cannot be an isomorphism by the first paragraph.

No two arrows $X_1\to X_\alpha$, $X_1\to X_\beta$ for $\alpha<\beta$ can be identical, as they would lead to a diagram $X_1\to X_\alpha\to X_\beta$ contradicting our assumptions.

This leads to a contradiction, as the arrows $X_1\to?$ form a set.
\end{proof}

\begin{proof}[Proof of Theorem~\ref{ThmVic}]
By Lemmata~\ref{LemSmall} and~\ref{LemHF}, the category $\cS(\bT)$ satisfies the conditions in Lemma~\ref{LemZorn}. Hence $\cS(\bT)$ contains an object $\bT\tto \bT_1$ so that every surjective tensor functor out of $\bT_1$ is an equivalence.
\end{proof}

\begin{remark}
Much of the work in proving Theorem~\ref{ThmVic} originates from the fact that surjective tensor functors are not necessarily epimorphisms in $\PTann_{\bk}.$ Indeed, otherwise one could use the standard Zorn's lemma rather than Lemma~\ref{LemZorn}. A concrete example of a surjective tensor functor which does not yield an epimorphism is $\Rep_{\bk}S_3\to\Rep_{\bk}\mZ/3$ (induced from an inclusion $\mZ/3\hookrightarrow S_3)$, see Example~\ref{ExInnAut}. The latter example, in case $\mathrm{char}(\bk)=3$, even shows that a tensor functor which is essentially surjective need not be an epimorphism.

\end{remark}

\begin{corollary}\label{CorFam}
For any (small) family $\{\bT_i, i\in I\}$ of incompressible categories in $\MdGr_{\bk}$, there exists an incompressible category $\bT_I$ in $\MdGr_{\bk}$ which contains every $\bT_i$ as a subcategory.
\end{corollary}
\begin{proof}
By definition of incompressible categories, it is sufficient to show there exist tensor functors $\bT_i\to \bT_I$ for some incompressible category $\bT_I$.

We prove the claim by transfinite induction. Choose a well-order on $I$, so we can write
$$\bT_1,\bT_2,\bT_3,\cdots,\bT_\omega,\bT_{\omega+1},\cdots, \bT_\beta.$$
For each $1\le i\le \beta$ we define $\bT_i'$ as follows. We set $\bT_1'=\bT_1$. If we have defined $\bT'_\gamma$, we define $\bT'_{\gamma+1}$ as some incompressible category to which we have a tensor functor out of $\bT'_\gamma\boxtimes \bT_{\gamma+1}$, which exists by Theorem~\ref{ThmVic}. If $\gamma$ is a limit ordinal, then we set $\bT'_\gamma=\varinjlim_{a<\gamma} \bT'_{a}$, which is well-defined by Proposition~\ref{PropFiltLim}(1) and Remark~\ref{RemLimInc}(3).

We can then take $\bT_I:=\bT_\beta'$.
\end{proof}

\begin{corollary}\label{CorCard}
For each infinite cardinality $\alpha$, there is an incompressible $\bT^\alpha\in\MdGr_{\bk}$ such that every tensor category in $\MdGr_{\fk}^{<\alpha}$ admits a tensor functor to~$\bT^\alpha$.
\end{corollary}
\begin{proof}
By Lemma~\ref{LemSmall0}, the isomorphism classes of incompressible categories in $\MdGr_{\fk}^{<\alpha}$ form a set. By Theorem~\ref{ThmVic}, every category in $\MdGr_{\fk}^{<\alpha}$ admits a tensor functor to one of these categories. We can thus apply Corollary~\ref{CorFam} to this family of incompressible categories.
\end{proof}

\begin{example}
We cannot always expect to have $\bT^\alpha\in \MdGr^{<\alpha}_{\fk}$. For $\mathrm{char}(\fk)=p$ and $\alpha=\aleph_0$, the natural candidate for $\bT^{\aleph_0}$ is $\Ver_{p^\infty}=\varinjlim \Ver_{p^n}\in \MdGr^{<\aleph_1}_{\fk}$.
\end{example}

\subsection{Incompressible and subterminal categories}

\begin{prop}\label{incoo}
\begin{enumerate}
\item A subterminal Bezrukavnikov pretannakian category is incompressible.
\item A tensor subcategory of a subterminal (resp. Bezrukavnikov) pretannakian category is also subterminal (resp. Bezrukavnikov).
\item If $\mathrm{char}(\fk)=0$, the only subterminal pretannakian categories are $\Vecc$ and $\mathsf{sVec}$.
\end{enumerate}

\end{prop}
\begin{proof}
For part (1), consider a subterminal Bezrukavnikov category $\bD$, and a surjective tensor functor $\bD\to \bT$. By assumption, there exists a tensor functor $\bT\to \bD$. Now by uniqueness, the composite $\bD\to\bT\to\bD$ is an equivalence (isomorphic to the identity). This forces $\bD\to\bT$ to be injective too and hence an equivalence.

For part (2), clearly subcategories of subterminal categories are subterminal, whereas the Bezrukavnikov property is inherited due to Remark~\ref{RemSq}.

For part (3), we assume $\mathrm{char}(\fk)=0$ and assume that $\cC$ is subterminal. First we show that $\cC$ is of moderate growth. In characteristic zero, it is easy to show that any extension~$E$ of simple objects of moderate growth remains of moderate growth (for instance, using \cite{Del02}, we can construct a simple algebra $A$ in $\Ind\cC$ so that $A\otimes E$ has a filtration by copies of $A$ and $A\otimes \bar{\unit}$). So if $\cC$ is not of moderate growth, it has a simple object $X$ which is not of moderate growth, say of dimension $t\in\bk$. By uniqueness of $(\Rep GL)_t^{ab}\to\cC$, we find $X\simeq X^\vee$. Similarly, by uniqueness of objects of dimension $t(t+1)/2$, we find
$$\Sym^2 X\;\simeq\; \Lambda^2 X\oplus X.$$
Now this contradicts
$$\Hom(\unit,X\otimes X)\simeq\End(X)\simeq \bk\quad\mbox{and}\quad \Hom(\unit,X)=0.$$

So $\cC$ is of moderate growth and thus by Deligne's theorems~\cite{Del90, Del02}, $\cC$ is tannakian or supertannakian. For a tannakian category $\Rep G$ with $G$ non-trivial we have two non-isomorphic tensor endofunctors: the identity functor and the composition 
$$
\Rep G\to\Vecc\to \Rep G.
$$
To deal with categories $\Rep (G,\phi)$ for supergroups $G$, we observe the following. Let $V\in \Rep (G,\phi)$ be a representation such that $G$ does not act trivially on $V_{\bar{0}}$. Consider the supergroup  $G \ltimes \mA_V$. Then $G$ is a subgroup of $G\ltimes \mA_V$, see for instance \cite[7.4.1]{ComAlg}. For any non-invariant $v\in V_{\bar{0}}$, we have a second homomorphism $G\to G\ltimes \mA_V$, $g\mapsto (g,gv-v)$. These lead to two non-isomorphic tensor functors to $\Rep(G,\phi)$. So if $\Rep(G,\phi)$ is subterminal, $G$ must act trivially on every $V_{\bar{0}}$ for every $V\in \Rep(G,\phi)$. By considering $(V_{\bar{1}})^{\otimes 2}\subset (V^{\otimes 2})_{\bar{0}}$, it then follows quickly that $\phi:\mZ/2\to G$ is actually an isomorphism (or $G$ is trivial) and so $\Rep(G,\phi)=\sVec$ (or $\Rep(G,\phi)=\Vecc$).
\end{proof}



\begin{theorem}\label{ThmFilF1}
The following statements are equivalent, for a fixed $\fk$:
\begin{enumerate}
\item $\MdGr_{\fk}$ is filtered;
\item Every incompressible category in $\MdGr_{\fk}$ is subterminal.
\item Every incompressible category in $\MdGr_{\fk}$ is subterminal and Bezrukavnikov.
\end{enumerate}
\end{theorem}
\begin{proof}
We observe first that (1) implies (3).  Let $\bD\in\MdGr_{\bk}$ be incompressible. By assumption in (1), for a pair of tensor functors $\bT\rightrightarrows \bD$ we can consider a tensor functor $F:\bD\to \bT_1$ which equalises them. However, since $F$ is an inclusion of a full subcategory, the parallel pair of functors were already isomorphic, showing that $\bD$ is subterminal. For the set-up in Definition~\ref{Defs}(3), by assumption there must be a commutative square
$$\xymatrix{
\bT_`\ar@{->>}[rr]^F\ar[d]_G&&\bT_2\ar@{-->}[d]\\
\bD\ar@{-->}[rr]&&\bT'.
}$$
Since $\cD$ is incompressible, the lower horizontal arrow is the inclusion of a tensor subcategory, so the conclusion follows from Remark~\ref{RemSq}.

That (3) implies (2) is trivial. 

Now we show that (2) implies (1). Since coproducts in $\MdGr_{\bk}$ are given by Deligne's tensor product, it suffices to prove that any two parallel arrows can be equalised. For such parallel arrows $\cC_1\rightrightarrows\cC_2$, we can use existence of a tensor functor $\cC_2\to\cD$ to an incompressible category $\cD$, see Theorem~\ref{ThmVic}. By assumption, $\cD$ is subterminal, so the two composite functors $\cC_1\rightrightarrows\cD$ must be isomorphic.
\end{proof}

\begin{remark}
Even though coproducts and filtered colimits exist in $\MdGr_{\bk}$, it cannot be cocomplete. For instance, it follows easily that the diagram 
$$\Rep_{\bk}\mZ/3\;\leftarrow\; \Rep_{\bk}S_3\;\to\;\Rep_{\bk}\mZ/3,$$
induced from inclusions $\mZ/3\hookrightarrow S_3$, see Example~\ref{ExInnAut}, does not have a pushout.
\end{remark}

\begin{remark}
\label{F1rem}
One could say that a pretannakian category $\cD$ is `weakly' subterminal if for every pretannakian category $\bT$, up to composition with an auto-equivalence of $\cD$, there is at most one tensor functor $\bT\to\bD$.

Clearly weakly subterminal Bezrukavnikov categories must still be incompressible. However, by Theorem~\ref{ThmFilF1}, the existence of an incompressible category which is not (strictly) subterminal, would imply $\MdGr_{\bk}$ is {\bf not} filtered. Furthermore, the weakly subterminal condition is not obviously inherited by tensor subcategories and not directly connected to conditions {\bf MN} and {\bf GR}. 
\end{remark}

%
%
%
%


\section{Exact ind-algebras}\label{SecAlgebras}
Let $\cC$ be a pretannakian category.
\subsection{Definition and results}

\subsubsection{} For $A\in\Alg\bT$, we can consider the symmetric monoidal category $\Mod_{\cC}A$ of $A$-modules in $\Ind\cC$, with tensor product $-\otimes_A-$.

We say that $A\in\Alg\cC$ is connected if it cannot be written as $A=A_1\times A_2$ for non-zero algebras $A_1,A_2$. Note that this is equivalent to the ring $A^{\inv}:=\Hom(\unit,A)$ being connected in the classical sense. Simple algebras are obviously connected.

\begin{prop}\label{PropExact}
The following two conditions are equivalent on $A\in\Alg\cC$:
\begin{enumerate}
\item We can write $A$ as a finite product $\prod_i A_i$ of algebras $A_i$ for which $\Mod_{\cC}A_i$ is the ind-completion of a pretannakian category over $\bk$ (in particular, $\prod_i A_i$ is the unique decomposition of $A$ into connected algebras).
\item \begin{enumerate}
\item $A$ is absolutely flat ({\it i.e.} the functor $-\otimes_A-$ is bi-exact on $\Mod_{\cC}A$);
\item $\dim_{\bk}\Hom(\unit,A)<\infty$;
\item $A$ is artinian ({\it i.e.} the $A$-module $A\otimes X$ is of finite length for every $X\in\cC$).
\end{enumerate}
\end{enumerate}
If these properties are satisfied, we say that $A$ is {\bf exact}.
\end{prop}
\begin{proof}
First we show that (1) implies (2). It suffices to consider the case of connected $A$. In this case $\Mod_{\cC}A$ is the ind-completion of a pretannakian category over $\bk$. This implies that the tensor product is exact and that
$$\bk\simeq \End_A(A)\simeq \Hom(\unit,A).$$
Finally, since objects in a pretannakian category have finite length (and a pretannakian category is a Serre subcategory of its ind-completion), the $A$-module $A\otimes X$ has finite length.

Now we prove that (2) implies (1). Let $\cC_A$ be the full monoidal subcategory of $\Mod_{\cC}A$ comprising finitely generated $A$-modules. By (2)(c) this is also the category of finite length $A$-modules, hence a Serre subcategory of $\Mod_{\cC}A$, and satisfies $\Ind\cC_A\simeq\Mod_{\cC}A$.
The monoidal category $\cC_A$ is also easily seen to be rigid, see for instance \cite[Proposition~2.36]{BEO}. It thus follows from condition (2)(b) and \cite[Proposition~1.17]{DM} that $A$ is a finite product of simple algebras $A_i$ with $\Hom(\unit, A_i)=\bk$ (since $\bk$ is algebraically closed).  Thus $\cC_{A_i}$ is a tensor category over $\bk$ and pretannakian by (2)(c).
\end{proof}

\begin{remark}
\begin{enumerate}
\item We refer to algebras $A$ as in Proposition~\ref{PropExact} as {\bf exact ind-algbras}. We reserve the term `exact algebra' for algebras which belong to $\cC$ itself (where conditions (2)(b)-(c) are automatic).
\item It follows quickly from Proposition~\ref{PropExact} that exact ind-algebras in a finite tensor  category $\cC$ are automatically contained in $\cC$, so they are exact algebras.
\item Since unit objects are simple in tensor categories, Proposition~\ref{PropExact} implies in particular that every connected exact ind-algebra is simple.
\item For a connected ({\it i.e.} simple) exact ind-algebra, the tensor category of which $\Mod_{\cC}A$ is the ind-completion is the category $\cC_A$ of finitely generated (or finitely presented) modules. In case $A\in\cC$, this is simply the category of $A$-modules in $\cC$.
\end{enumerate}
\end{remark}

\begin{question}\label{NQ}
Is every simple algebra $A\in\Alg\cC$ with $\dim_{\bk}\Hom(X,A)<\infty$ for all $X\in\cC$ (or more generally, simply with $\dim_{\bk}\Hom(\unit,A)<\infty$) automatically exact?
\end{question}

\begin{theorem}\label{ThmSumm}

\begin{enumerate}
\item Consider $\bT=\Rep G$ for an affine group scheme $G$ over $\fk$. 
A connected $A\in\Alg\cC$ is exact if and only if $A\simeq \cO(G/H)$ for a subgroup $H<G$ for which  the induction functor from $H$ to $G$ is exact and faithful.

\item If $\bT$ is a finite tensor category then $A\in\bT$ is exact if and only if for every projective $P\in \cC$ and $A$-module $M$ in $\cC$, the $A$-module $M\otimes P$ is projective ({\it i.e.} $A$ is exact according to \cite[Chapter 7]{EGNO}).

\item \'Etale algebras (defined in Theorem~\ref{ThmEta}), are exact.
\end{enumerate}
\end{theorem}

\begin{remark}
The main result of \cite{CPS} (see \cite[Theorem~A.2.1]{ComAlg} for the required generality) states that, if $G$ is of finite type, then the closed subgroups $H<G$ in \ref{ThmSumm}(1) are precisely those for which the quotient scheme $G/H$ is affine.
\end{remark}

Our main motivation to study exact algebras is the following obvious lemma.
\begin{lemma}\label{LemIncEx}
If $\cC$ is incompressible, the only connected exact ind-algebra in $\cC$ is $\unit$.
\end{lemma}
 In particular, we derive the following consequences of Theorem~\ref{ThmSumm}.

\begin{corollary}\label{CorEx}
\begin{enumerate}
\item A finite tensor category $\bT$ is incompressible if and only if $\unit$ is the only connected exact algebra in $\bT$.
\item A tensor subcategory of an incompressible finite tensor category is incompressible.
\end{enumerate}
\end{corollary}

\begin{remark}
\begin{enumerate}
\item It is conjectured in \cite[Conjecture~12.6]{EOf} that every simple algebra (not necessarily commutative) in a finite tensor category is exact. If this conjecture it correct, it would thus imply that {\em a finite tensor category is incompressible if and only if $\unit$ is the only simple algebra in $\cC$}. This conjecture would of course also provide an affirmative answer to the weak form of Question~\ref{NQ}.
\item Analysing the proof of Proposition~\ref{PropF1} below shows that any finite tensor category in which $\unit$ is the only simple algebra is subterminal in the category of finite tensor categories. The conjecture in part (1) would thus further imply that every finite incompressible tensor category is subterminal in this sense.
\end{enumerate}

\end{remark}

\begin{corollary}\label{CorEta}
Let $\bT$ be a pretannakian category.
\begin{enumerate}
\item If $\bT$ is incompressible, then the only \'etale algebras in $\bT$ are ordinary $\fk$-\'etale algebras ({\it i.e.} finite products of $\unit$).
\item For every $\fk$, there exist pretannakian categories over $\fk$ that do not admit tensor functors to incompressible categories.
\item Assume $\mathrm{char}(\fk)=0$.
The categories $(\Rep S)_t^{ab}$ of \cite{CO, Del07}, $t\in \fk$, do not admit tensor functors to incompressible categories.
\end{enumerate}
\end{corollary}

The remainder of this section is devoted to the proof of the theorem and its corollaries.

%
%

\subsection{Tannakian categories}

We start with the following general standard lemma. For an algebra $A\in\Alg\cC$, we let $\cC_A\subset\Mod_{\cC}A$ denote the full subcategory consisting of finitely presented $A$-modules.
\begin{lemma}\label{LemBeck}
The connected exact ind-algebras in $\cC$ are precisely those $A\in \Alg\cC$ which are of the form $A=F_\ast\unit$ for some tensor functor $F:\cC\to\cC'$ with $F_\ast$ exact and faithful. Moreover, for such an $F$, with $A:=F_\ast \unit$, we have an equivalence $\cC'\simeq \cC_A$ yielding a commutative triangle
$$\xymatrix{
\cC\ar[rr]^F\ar[rrd]_{A\otimes-}&&\cC'\ar[d]^{\simeq}\\
&&\cC_A
}$$
\end{lemma}

\begin{proof} We start with the second claim. Consider a tensor functor $F:\cC\to\cC'$ with $F_\ast$ exact and faithful.
By Barr-Beck monadicity, $F_\ast$ lifts to an equivalence between $\Ind\bT'$ and the category of $F_\ast F $-modules in $\Ind\bT$. Using adjunction, we find that the natural transformation $A\otimes-\Rightarrow F_\ast F$ of monads is an isomorphism. The resulting equivalence 
$$\Ind\cC'\;\xrightarrow{\sim}\;\Mod_{\cC}A$$ is easily seen to be monoidal.

In particular this shows that $F_\ast\unit$ is exact. Conversely, starting from a connected exact algebra (in which case $\cC_A$ is a tensor category with ind-completion $\Mod_{\cC}A$), the adjoint of $\cC\to\cC_A$ is just given by the (exact and faithful) forgetful functor $\cC_A\to\Ind\cC$, concluding the proof.
\end{proof}

\begin{remark}\label{RemFaith}
If for some tensor functor $F:\cC\to\cC'$ the adjoint $F_\ast$ is faithful, then $F$ is surjective, since the counit of the adjunction evaluates to epimorphisms $FF_\ast X\to X$.
\end{remark}

\begin{proof}[Proof of Theorem~\ref{ThmSumm}(1)]
By Lemma~\ref{LemBeck} we need to find all $F_\ast\unit$ for tensor functors $F:\Rep G\to \cC'$ with $F_\ast$ faithful and exact. By Remark~\ref{RemFaith} and the fact that $\Vecc$ is Bezrukavnikov (and Example~\ref{ExGH}), it follows that we need to take tensor functors $F:\Rep G\to \Rep H$ given by restriction to a subgroup $H<G$, for which $F_\ast=\Ind^G_H$ is faithful and exact.
\end{proof}

\subsection{Finite categories}


\begin{proof}[Proof of Theorem~\ref{ThmSumm}(2)]
Consider a finite tensor category $\bT$ with an algebra $A\in\bT$. Without loss of generality, we will assume that $A$ is connected.

Assume first that $A$ is absolutely flat, so that $\cC_A$ is a tensor category.  Then for any projective $P\in\bT$ and object $M\in\bT_A$, we have the isomorphisms
$$\Hom_A( M\otimes P,-)\;\simeq\;\Hom_A(A\otimes P,M^\vee\otimes_A-)\;\simeq\;\Hom(P, M^\vee\otimes_A-).$$
As the right-hand side is an exact functor, we conclude that $M\otimes P$ is projective, so $A$ is exact.

Now assume that $A$ is exact in the terminology of \cite[Chapter~7]{EGNO}. That $A$ is absolutely flat follows from the isomorphism (again for any projective $P\in\bT$ and $A$-module $M$ in $\bT$) 
$$(M\otimes P)\otimes_A-\;\simeq\; P\otimes (M\otimes_A-)$$
of functors $\cC_A\to\cC$. Indeed, since $N\otimes P$ is projective, the functors are exact, forcing $M\otimes_A-$ to be exact too.
\end{proof}

%

 \begin{lemma}\label{PropSurjFin}
 If $\bT$ is a finite tensor category, and $F:\bT\to\bT'$ a surjective tensor functor, then $F_\ast$ is faithful and exact.
\end{lemma}

\begin{proof}
It follows that $\bT'$ is also finite. That $F_\ast$ is exact follows from
$$\Hom(P,F_\ast-)\simeq \Hom(FP,-)$$
for each projective $P\in\bT$, since $FP$ is projective in $\bT'$ by \cite[Thm~6.1.16]{EGNO}. By surjectivity and the fact that projective objects are injective, every projective object in $\bT'$ is a direct summand of some $FP$, from which faithfulness of $F_\ast$ follows. 
\end{proof}

\begin{proof}[Proof of Corollary~\ref{CorEx}]
One direction of part (1) is in Lemma~\ref{LemIncEx}.
For the other direction, assume that $\bT$ is not incompressible. We thus have a surjective tensor functor $\bT\to \bT'$  which is not an equivalence. By~\autoref{PropSurjFin} and Lemma~\ref{LemBeck}, there exists a simple exact algebra $A\not=\unit$ in $\cC$.

Part (2) follows immediately from part (1) and \cite[Corollary~12.4]{EOf}, which states that exact algebras in a finite tensor category remain exact when interpreted in a larger finite tensor category.
\end{proof}

We conclude with some direct consequences of the previous results.
\begin{lemma}\label{IncProd}
The following conditions are equivalent on finite tensor categories  $\cD_1$, $\cD_2$:
\begin{enumerate}
\item $\cD_1\boxtimes\cD_2$ is incompressible;
\item $\cD_1$ and $\cD_2$ are both incompressible and the only tensor category which is both a tensor subcategory of $\cD_1$ and $\cD_2$ is $\Vecc$.
\end{enumerate}
\end{lemma}
\begin{proof}
First we show that (1) implies (2). That $\cD_1$ and $\cD_2$ are incompressible follows either directly or from applying Corollary~\ref{CorEx}(2) to $\cD_i\subset \cD_1\boxtimes\cD_2$. Let $\cD$ be a tensor category which we can view as a tensor subcategory of both $\cD_1$ and $\cD_2$. Again by Corollary~\ref{CorEx}(2) (and Example~\ref{ExDTP}), $\cD\boxtimes\cD$ is incompressible. By considering the tensor product $\cD\boxtimes\cD\to\cD$, it follows that we must have $\cD\simeq\Vecc$.

Now we show that (2) implies (1). Consider a tensor functor
$$\cD_1\boxtimes\cD_2\to\cC.$$
We need to prove that it is injective, meaning full by Remark~\ref{RemInj}. The functor is determined by two (full by assumption) tensor functors $F_i:\cD_i\to\cC$. Since every object in $\cD_1\boxtimes\cD_2$ is a subobject and a quotient of one in $\cD_1\otimes\cD_2$, it is sufficient to show that
$$\Hom(X,Y)\otimes \Hom(V,W)\to \Hom(F_1(X)\otimes F_2(V), F_1(Y)\otimes F_2(W))$$
is an isomorphism for all $X,Y\in\cD_1$ and $V,W\in\cD_2$. By adjunction we can assume $Y=\unit=V$, so the problem reduces to the showing that
$$\Hom(F_1(X),\unit)\otimes_{\bk}\Hom(\unit, F_2(W))\;\simeq\;\Hom(F_1(X),F_2(W)).$$
Since $F_1,F_2$ are embeddings of tensor subcategories, the above is equivalent to the claim that $\cD_1\cap\cD_2=\Vecc$, as subcategories of $\cC$.
\end{proof}

\begin{remark}
We expect Lemma~\ref{IncProd} to remain true for arbitrary pretannakian categories. One case which is easy to verify is that, for a pretannakian category $\cD$, the category $\cD\boxtimes\sVec$ is incompressible if and only if $\cD$ is incompressible and does not contain $\sVec$.
\end{remark}

\begin{remark}
If $\cD$ is a finite incompressible tensor category, then its maximal pointed tensor subcategory is either $\Vecc$ or $\sVec$ (to be interpreted as $\Ver_4^+$ if $p=2$). 
Indeed, by Corollary~\ref{CorEx}(2), this subcategory is again incompressible, so it follows from the fact that pointed tensor categories admit a fibre functor to $\sVec$, see \cite{EG1, EG2}.
\end{remark}

\subsection{\'Etale algebras}

For an algebra $(A,\mu,\eta)$ in $\bT$, we have the trace $\Tr:A\to\unit$ given by
$$\Tr:A\xrightarrow{A\otimes \co_A} A\otimes A\otimes A^\vee\xrightarrow{\mu\otimes A^\vee} A\otimes A^\vee \xrightarrow{\ev_A\circ \sigma_{A,A^\vee}}\unit. $$
We denote the (symmetric) trace form by $\tau:=\Tr\circ\mu:A\otimes A\to\unit$.

\begin{theorem}\label{ThmEta} Let $\bT$ be a pretannakian category.
For a (commutative) algebra $(A,\mu,\eta)$ in $\Ind\bT$, the following conditions are equivalent:
\begin{enumerate}
\item $A$ is separable, meaning there exists a section $s:A\to A\otimes A$ of $\mu$ as $A$-bimodules.
\item $A\in\bT$ and there exists $\alpha:\unit\to A\otimes A$ with the property that
$$(\id_A\otimes \tau)\circ (\alpha\otimes \id_A)=\id_A.$$
\item $A$ is special Frobenius, meaning $A\in\bT$ and the following is an isomorphism:
$$(\tau\otimes \id_{A^\vee})\circ (\id_A\otimes \co_A):A\to A^\vee.$$
\end{enumerate}
If $A$ satisfies one of these equivalent conditions, we say it is {\bf \'etale}.
\end{theorem}

\begin{proof}
By considering the image of $\alpha$ under
$$\Hom(\unit,A\otimes A)\simeq \Hom(A^\vee,A),$$
it follows that the condition in (2) is equivalent to the displayed morphism in (3) being a retract (which is equivalent to being an isomorphism, as follows for instance from the finite length assumption, or from Krull-Schmidt). Hence (2) and (3) are equivalent.

Now we prove that (1) implies (2). For a section $s$, we define $\alpha: s\circ\eta:\unit\to A\otimes A$. First we argue that $A\in\bT$. By compactness of $\unit$, the morphism $\alpha$ lands in $A'\otimes A''$ for subobjects $A',A''\in\bT$ of $A$. Now using that $s$ is a morphism of left (resp.) right $A$-modules shows that the image of $s$ is contained in $A\otimes A''$ (resp. $A'\otimes A$). In conclusion, $s$ corresponds to a morphism $A\to A'\otimes A''$ and $A$ is a direct summand of $A'\otimes A''\in\bT$.

By definition of $s$, we have
\begin{equation}\label{eq1}\eta\;=\;\mu\circ \alpha.\end{equation}
From the fact that $s$ is a right and left $A$-module morphism, we find the two relations:
\begin{equation}\label{mualpha}(A\otimes \mu)\circ (\alpha \otimes A)\;=\; s\;=\; (\mu\otimes A)\circ (A\otimes \alpha).\end{equation}
Next we calculate, using the definition of $\Tr$, then the equality between the left and right term in \eqref{mualpha}, then commutativity of $\mu$ and finally \eqref{eq1}:
\begin{eqnarray*}(A\otimes \Tr)\circ \alpha &=&(A\otimes \widetilde\ev_A)\circ (A\otimes \mu\otimes A^\vee)\circ(\alpha\otimes \co_A)\\
&=& (\mu\otimes \widetilde\ev_A)\circ (A\otimes \alpha\otimes A^\vee)\circ\co_A\\
&=& (\mu\otimes \ev_A)\circ (A\otimes \co_A\otimes A)\circ\alpha=\mu\circ\alpha\\
&=& \eta.
\end{eqnarray*}
From that result we find
$$(\mu\otimes \Tr)\circ (A\otimes\alpha)=\id_A.$$
By applying \eqref{mualpha} again, it follows that this can be rewritten as the identity in (2).

Finally, we prove that (3) implies (1). Denote by $\phi:A^\vee\to A$ a (two-sided) inverse to the morphism in (3). We will prove that
$$s:=(\mu\otimes \phi)\circ(A\otimes \co_A):A\to A\otimes A$$
is a section of $\mu$ as bimodules. We start with the following calculation, using the definition of $\tau$, the property of $\phi$ as an inverse and finally the snake relation:
\begin{eqnarray*}
\tau\circ (\eta\otimes A)&=& \widetilde{\ev}_A\circ(\mu\otimes A^\vee)\circ(A\otimes \co_A)\\
&=&\widetilde{\ev}_A\circ(A\otimes \tau\otimes A^\vee)\circ(A\otimes A\otimes \co_A) \circ(\mu\otimes \phi)\circ(A\otimes \co_A)\\
&=&\tau \circ(\mu\otimes \phi)\circ(A\otimes \co_A).
\end{eqnarray*}
Now we use the property of $\phi$ as an inverse, the above equality, the associativity and finally symmetry of $\tau$ and again the property of $\phi$ as an inverse to calculate
\begin{eqnarray*}
\eta&=& (\tau\otimes\phi)\circ(\eta\otimes\co_A)\\
&=&(\tau\otimes \phi)\circ(\mu\otimes \phi\otimes A^\vee)(A\otimes \co_A\otimes A^\vee)\circ\co_A\\
&=&(\tau\otimes \phi)\circ(A\otimes \mu\otimes A^\vee)\circ(A\otimes A\otimes \phi\otimes A^\vee)(A\otimes \co_A\otimes A^\vee)\circ\co_A\\
&=& \mu\circ (A\otimes\phi)\circ \co_A.
\end{eqnarray*}
Using the above relation together with associativity of $\mu$ then finally shows that $s$ is a section of $\mu$.

By construction, $s$ is a morphism of left $A$-modules. To complete the proof, we show that $s$ is also a morphism of right $A$-modules. By first applying the property of $\phi$ as a right inverse, then associativity of $\mu$ and finally, the adjoint of the property of $\phi$ as a left inverse, we calculate
\begin{eqnarray*}
&&(A\otimes\mu)\circ (A\otimes\phi\otimes A)\circ (\co_A\otimes A)\\
&=&(A\otimes \tau\otimes \phi)\circ(A\otimes A\otimes\co_A)\circ(A\otimes\mu)\circ (A\otimes\phi\otimes A)\circ (\co_A\otimes A) \\
&=&(A\otimes \tau\otimes A)\circ (A\otimes\phi\otimes \mu\otimes \phi)\circ (\co_A\otimes A\otimes \co_A) \\
&=&(\mu\otimes\phi)\circ (A\otimes \co_A).
\end{eqnarray*}
That $s$ is a morphism of right $A$-modules is an easy consequence of the above relation.
\end{proof}

\begin{remark}
\begin{enumerate}
\item The proof of Theorem~\ref{ThmEta} remains valid in any Krull-Schmidt additive rigid symmetric monoidal category.
\item In symmetric (even braided) fusion categories, \'etale algebras were defined and studied in \cite{DMNO}.
\end{enumerate}

\end{remark}
\begin{corollary}\label{CorEta1}
Consider a tensor functor $F:\bT\to\bT'$ between pretannakian categories and an algebra $A\in \bT$. Then $A$ is \'etale if and only if $F(A)$ is \'etale.
\end{corollary}
\begin{proof}
This follows immediately from Theorem~\ref{ThmEta}(3).
\end{proof}

\begin{proof}[Proof of Theorem~\ref{ThmSumm}(3) and Corollary~\ref{CorEta}]
Let $A\in \cC$ be an \'etale algebra and $X$ an $A$-module with underlying object $X_0\in\bT$. It follows from the definition in Theorem~\ref{ThmEta}(1) that, as $A$-modules,
$$X\simeq A\otimes_A X\quad\mbox{is a direct summand of}\quad (A\otimes A)\otimes_AX\simeq A\otimes X_0.$$
This shows that $A$ is absolutely flat and hence exact.


Corollary~\ref{CorEta}(1) is an obvious consequence of the above and Lemma~\ref{LemIncEx}.

Now consider a pretannakian category $\bT$ with an \'etale algebra $A$ which is not of moderate growth (as an object in $\bT$). By Corollary~\ref{CorEta1}, under any tensor functor $F:\bT\to \bT'$, the algebra $F(A)$ is an \'etale algebra which cannot be a product of copies of $\unit$. Hence $\bT'$ is not incompressible, by Corollary~\ref{CorEta}(1).

This implies Corollary~\ref{CorEta}(3), and similar examples in positive characteristic can be derived from \cite[Section~3.3]{Ha}, proving Corollary~\ref{CorEta}(2).
\end{proof}

%

\begin{proof}[Conclusion of the proof of \autoref{PropFilt}]
Assume for a contradiction that we have a tensor functor 
$$F:(\Rep S)_t \boxtimes \sVec\;\to\; \bT$$
which equalises $F_1,F_2$. Since $A$ is an \'etale algebra, it follows from Theorem~\ref{ThmEta} that $A$ and therefore also $C:=F(A)\simeq F(B)$ is separable in $\bT$. The observation that $\bT_C$ is therefore an abelian rigid monoidal category is still valid (even though $\bT$ might not be artinian). However, $C$ has a subobject $F(\bk^{0|1})$ with vanishing symmetric power. This thus generates a nilpotent ideal $I<C$ (in fact $I^2=0$). This is in particular a subobject (of the tensor unit) in $\bT_C$ which, by \cite[Proposition~1.17]{DM} and its proof, implies that we have a decomposition as $C$-modules 
$$C\;\simeq\; I\oplus I^{\perp},\quad \mbox{with}\quad I\otimes_C I^{\perp}=0.$$
However, this in turn implies 
$$I\simeq I\otimes_CC\simeq I^2+I\otimes_C I^{\perp}=0,$$
a contradiction.
\end{proof}

\begin{remark}
While the authors were finalising the manuscript some of the results on \'etale algebras appeared independently in \cite{HS}. In fact, the results {\em loc. cit.} and references therein imply that there exist a whole range of pretannakian categories which do not admit tensor functors to incompressible categories. Concretely, `discrete' pretannakian categories only admit tensor functors to incompressible categories if they are of moderate growth.
\end{remark}


\section{Maximally nilpotent tensor categories}\label{SecMN}

For a morphism $\alpha:\unit\to X$ in a tensor category, we denote by $\alpha^j$ the composition
$$\unit\xrightarrow{\alpha^{\otimes j}} X^{\otimes j}\tto\Sym^j X.$$
We let $\bT$ denote a pretannakian category.

\subsection{Symmetric growth dimension}

\subsubsection{} For $X\in \bT$ we define its {\bf symmetric growth dimension} $\sgd(X)$ as
$$\sgd(X)\:=\;\limsup_{n\to\infty}\frac{ \log(\ell (\Sym^{\le n}X))}{\log (n)}\;\in [0,+\infty]=\mR_{\ge 0}\cup\{\infty\}.$$

\begin{remark}\label{RemFP}
In the definition of $\sgd$ we can replace the length function $\ell$ with any other group homomorphism $\chi:K_0(\bT)\to\mR$ which sends the classes of simple objects to positive numbers. We denote the resulting function by $\sgd_{\chi}$.

\begin{enumerate}
\item If $\bT$ only has a finite number of simple objects up to isomorphism (or is a union of such tensor categories), then $\sgd_{\chi}=\sgd$ for all $\chi$.
\item In case $\chi$ is a ring homomorphism, $\sgd_{\chi}(X)$ can be viewed as the Gelfand-Kirillov dimension of the symmetric algebra of $X$ and enjoys additional natural properties. 
\item If $\bT$ is finite, the unique $\chi$ as in (2) is $\FPdim$. More generally, we can take $\chi=\gd$, Section~\ref{Secgd} below, in many cases.
\end{enumerate} 
\end{remark}
\begin{example}\label{Extd}
\begin{enumerate}
\item Using Remark~\ref{RemFP}, for $G$ a finite group and $V\in \Rep_{\bk}G$, we find
$$\sgd (V)\;=\; \dim_{\bk} V.$$
Similarly, if $\mathrm{char}(\bk)\not=2$ for a central element $z\in G$ of order 2 and $V$ in the supertannakian category $\Rep_{\bk}(G,\phi)$ for $\phi:\mZ/2\hookrightarrow G$ the inclusion of the subgroup generated by $z$, we find
$$\sgd(V)\;=\;\dim_{\bk}V_{\bar{0}}.$$
\end{enumerate}
\noindent More generally, for affine group schemes, for $\chi=\mathsf{gd}=\dim_{\bk}$, we clearly have 
$$\sgd_{\chi}(V)\;=\;\mathsf{gd}(V)\;=\;\dim_{\bk}V.$$ However, this is no longer true for $\sgd$:
 \begin{enumerate}
 \setcounter{enumi}{1}
\item
For $\bk=\mC$ and $V_n\in \Rep_{\mC}GL_n$ the tautological representation on $\mC^n$,
$$\sgd (V_n)\;=\;1, \quad\mbox{and}\quad  \sgd(V_n^{\otimes 2})=n.$$ 
\item For $\bk=\mC$ and $V_n\in \Rep_{\mC}O_n$ the tautological representation, by Howe duality we have $\sgd V_n=2$.
\item For $\mathrm{char}(\bk)=2$ and $V$ the natural representation of $SL_2$, we can calculate, using $\Sym^nV\simeq \nabla(n)$, the Weyl module, and 
$$\ell(\nabla(2n))=\ell(\nabla(n))+\ell(\nabla(n-1))\quad\mbox{and}\quad \ell(\nabla(2n+1))=\ell(\nabla(n))$$
that
$$\sgd(V)\;=\;\log_2(3)\;\approx\;1.585.$$
\end{enumerate}
In general, $\sgd(X)$ need not be finite:
\begin{enumerate}\setcounter{enumi}{4}
\item Assume $k=\mC$ and $t\in \mC\backslash\mZ$. For $X_t$ the generating object of $(\Rep GL)_t$ we have 
$$\sgd (X_t) =1,\quad\mbox{and}\quad \sgd(X_t^{\otimes 2})=\infty.$$
Indeed this follows either from (2) by interpolation, or by observing that $\ell(\Sym^n(X_t^{\otimes 2}))$ equals the number of partitions of $n$.
\end{enumerate}
\end{example}

\begin{question}
\begin{enumerate}
\item  Is $\sgd(X)$ always finite in tensor categories of moderate growth?
\item In finite tensor categories, do we have
$$\sgd(X)\;\le\; \FPdim(X)?$$
\item Can we replace $\limsup$ by $\lim$ in the definition of $\sgd$?
\item What is $\sgd(V_n)$ for $V_n$ the tautological $GL_n$-representation in positive characteristic? (The case $n=2=p$ follows from Example~\ref{Extd}(4).)
\end{enumerate}

\end{question}


\begin{lemma}\label{Lemtd} Consider a short exact sequence $0\to X\to Y\to Z\to 0$ in $\cC$.
\begin{enumerate}
\item\label{ZY} We have $\sgd(Z)\le \sgd(Y).$
\item\label{Zm} If $Z=\unit^m$ with $m\in\mN$, then the image of $\Sym (X)\to\Sym (Y)$ is finite if and only if
$\sgd(Y)=m.$ 
\end{enumerate}
\end{lemma}
\begin{proof}

Part (1) follows directly from the epimorphisms $\Sym^nY\tto\Sym^nZ$, so we focus on part (2). 
By Theorem~\ref{ThmSym} we have
$$\gr \Sym (Y)\;\simeq\; A\otimes \bk[x_1,\cdots, x_m]$$
with $A$ the image of $\Sym (X)\to \Sym (Y)$ and $\bk[x_1,\cdots, x_m]$ a polynomial algebra in $m$ variables. The conclusion then follows quickly.
\end{proof}

%
%
%

\begin{remark}
In a finite tensor category (or in general if we replace $\sgd$ with $\sgd_{\chi}$ as in Remark~\ref{RemFP}(2)) the above lemma can be strengthened. For instance, we have $\sgd(Y)\le \sgd(X)+\sgd(Z)$, with equality if the short exact sequence is split.
\end{remark}

%
%
%
%

\subsection{The nilpotent property}

For $X\in\cC$, denote by $\widetilde{X}$ its subobject which fits into the short exact sequence
$$0\to \widetilde{X}\to X\to\Hom(X,\unit)^\ast\otimes \unit\to0.$$

\begin{prop}\label{Thmtilde} For $X\in\bT$, we have
\begin{equation}\label{ineq}
\dim\Hom(X,\unit)\;\le\; \sgd(X). 
\end{equation}
Furthermore, the following properties on $X$ are equivalent:
\begin{enumerate}
\item The inequality \eqref{ineq} is an equality.
\item For every tensor functor $F:\bT\to\cC'$ to a pretannakian category $\cC'$, we have
$$\dim\Hom(FX,\unit)\;=\; \sgd(X). $$
So in particular, every $F$ yields an isomorphism $\Hom(X,\unit)\xrightarrow{\sim}\Hom(FX,\unit)$.
\item The image of $\Sym (\widetilde{X})\to\Sym (X)$ is finite.
\end{enumerate}
\end{prop}
\begin{proof}
The inequality \eqref{ineq} is a special case of Lemma~\ref{Lemtd}(\ref{ZY}). 
Clearly (2) implies (1), and that (1) implies (3) follows from Lemma~\ref{Lemtd}(\ref{Zm}). It thus suffices to prove that (3) implies (2).

We consider the short exact sequence in $\bT'$
$$0\to F\widetilde{X}\to FX\to \unit^{d}\to 0 $$
 with $d:=\dim\Hom(X,\unit)$. By assumption, the image of $\Sym (F\widetilde X)\to\Sym (FX)$ is finite, so
Lemma~\ref{Lemtd}(\ref{Zm}) implies $\sgd(FX)=d$. Now, let $m$ be either side of the proposed equality in condition (2), then we have (using \eqref{ineq} and faithfulness of $F$)
$$d=\dim\Hom(X,\unit)\;\le m\;\le \sgd(FX)=d,$$
demonstrating in particular the equality.
\end{proof}

\begin{theorem}\label{Cond4Fib}
The following conditions are equivalent on $\bT$:
\begin{enumerate}
\item $\sgd Y=\dim \Hom(Y,\unit)$, for all $Y\in \bT$;
\item $\cC$ is {maximally nilpotent}, meaning:
\begin{enumerate}
\item For every non-trivial simple $L\in\bT$, the algebra $\Sym (L)$ is finite, and
\item For every non-split $\unit\to X$ in $\bT$, $\unit\to\Sym^n X$ is zero for some $n\in\mN$.
\end{enumerate}
\item 
\begin{enumerate}
\item For every non-trivial simple $L\in\bT$, the algebra $\Sym (L)$ is finite, and
\item For every simple $L\in\bT$ and indecomposable extension
$$ 0\to \unit^m\to X\to L\to 0,$$
the image of $\Sym(\unit^m)\to \Sym (X)$ is finite.
\end{enumerate}
\item The image of $\Sym\widetilde{X}\to\Sym(X)$ is finite for every $X\in\bT$.
\end{enumerate}
\end{theorem}
\begin{proof}
That (1) and (4) are equivalent is part of \autoref{Thmtilde}. That (4) implies (2) follows by observing the latter consists of special cases of the former.
That (2) implies (3) (even that (2)(b) implies (3)(b)) is a direct exercise.

To conclude the proof, we show that (3) implies (4). Assume that we have counterexamples to property (4), then we can take such a counterexample $X\in\bT$ of minimal length. First we observe that the socle of $X$ can only consist of copies of $\unit$. Indeed, consider a simple subobject $L\not=\unit$, leading to a commutative diagram
$$\xymatrix{
L\ar@{=}[r]\ar@{^{(}->}[d]&L\ar@{^{(}->}[d]\\
\widetilde{X}\ar@{^{(}->}[r]\ar@{->>}[d]&X\ar@{->>}[r]\ar@{->>}[d] &\unit^d\ar@{=}[d]\\
\widetilde{Q}\ar@{^{(}->}[r]&Q\ar@{->>}[r]& \unit^d
}$$
with exact rows and columns. By assumption, the image of $\Sym(\widetilde{X})\to\Sym(X)$ is infinite. We consider the filtrations on the symmetric algebras corresponding to the vertical short exact sequences. Also $\gr\Sym\widetilde{X}\to\gr\Sym(X)$ must have infinite image. By Theorem~\ref{ThmSym}, this morphism can be written as a product of algebra morphisms of the form
$$A_1\otimes \Sym\widetilde{Q}\;\to\; A_2\otimes \Sym(Q),$$
where $A_2$ (and also $A_1$) is a finite algebra (as a quotient of $\Sym(L)$). This forces $\Sym\widetilde{Q}\to \Sym(Q)$ to have infinite image, contradicting minimality of $X$. We can now conclude the proof by following the same reasoning, but replacing $L$ by $\unit^m$, the radical of an indecomposable length $m+1$ subobject of $X$.\end{proof}


\begin{theorem}\label{LemNoSim}
Let $\bD$ be a maximally nilpotent pretannakian category.
\begin{enumerate}
\item Every tensor functor $F:\bD\to\bT$ is full.
\item If $\bD$ is a finite tensor category, it is incompressible.
\item If $\bD$ is a tensor subcategory of a pretannakian category $\bT$, then for every simple algebra $A$ in $\Ind\bT$ and $X\in \bD$, the map
$$\Hom(X,\unit)\otimes_{\bk}\Hom(\unit,A)\;\to\; \Hom(X,A)$$
is an isomorphism.
\item The only simple algebras in $\Ind\bD$ are field extensions $K/\bk$.
\item $\cD$ satisfies the Hilbert Basis Property.
\end{enumerate}
\end{theorem}
\begin{proof}
Part (1) follows from \autoref{Thmtilde}.
Part (2) follows from part (1) and Remark~\ref{RemLimInc}(2).

The morphism in part (3) is always injective. If it were not surjective, there would exist $\bD\ni X\subset A$ for which $\widetilde{X}\not=0$ (in other words $X\not\in\mathsf{Vec}$). However, by assumption, the image of the algebra morphism
$\Sym \widetilde{X}\to A$ is finite, showing that $\widetilde{X}\subset A$ generates a nilpotent ideal, a contradiction.

Part (4) is the special case of part (3) where $\bT=\bD$.

Part (5) is proved in \cite[Lemma~6.4.4]{ComAlg}, in fact the property in \ref{Cond4Fib}(2)(a) is sufficient.
\end{proof}

\subsection{Subterminal categories}

Following Deligne's treatment of $\bD=\sVec$, we can prove that maximally nilpotent geometrically reductive pretannakian categories are subterminal.

\begin{theorem}\label{ThmF1}
If a pretannakian category $\cD$ is {\bf MN} and {\bf GR}, then it is subterminal.
\end{theorem}

We start by proving the following `finite' version, see also \cite[Proposition~1.5.9]{Tann}.
\begin{prop}\label{PropF1}
For $\cD$ as in Theorem~\ref{ThmF1} and a finitely generated pretannakian category $\bT$, there exists at most one tensor functor $\bT\to\bD$, up to isomorphism.
\end{prop}
\begin{proof}
All necessary ingredients for the proof can be found in \cite{Del02}. For two tensor functors $F,G:\bT\to\bD$,
we consider the Hopf algebra
$$\Lambda(F,G):=\int^{X\in \bT}F(X)^\vee\otimes G(X)$$
from \cite[\S 3]{Del02} or \cite{Del90}. A direct consequence of \cite[Lemma~3.6]{Del02} is that $\Lambda(F,G)\not=0$.

It is immediate to see that for $M\in\Ind\cD$, there is a bijection between morphisms $\Lambda(F,G)\to M$ and natural transformations $G\Rightarrow M\otimes F(-)$. Moreover, algebra morphisms $\Lambda(F,G)\to A$ correspond to natural transformations $A\otimes G\Rightarrow A\otimes F$ of monoidal functors $\cC\to\Mod_{\cD}A$. The latter transformations are automatically automorphisms. If we let $\mG_{F,G}$ denote the affine group scheme in $\cD$ corresponding to the Hopf algebra $\Lambda(F,G)$, the above thus interprets $\mG_{F,G}(A)$ for all $A\in\Alg\cC$.

In particular, $F$ and $G$ are isomorphic if and only if there exists an algebra morphism
$$\Lambda(F,G)\;\to\; \unit.$$
 By taking a quotient of $\Lambda(F,G)$ with respect to a maximal ideal, and Theorem~\ref{LemNoSim}(4), it therefore suffices to demonstrate that $\Lambda(F,G)$ is finitely generated as an algebra.

The latter follows from the arguments in the proof of \cite[Proposition~4.1]{Del02}: Firstly, if we let $Z\in\bT$ be a generator, then the homomorphism
$$\mG_{F,F}\;\to\; GL(FZ)$$
which evaluates a natural transformation at $Z$ is a monomorphism. We thus know from \cite{ComAlg}, see \ref{SecFaith}, that $\cO(GL(FZ))\to\Lambda(F,F)$ is an epimorphism in $\Ind\cC$ (not just in $\Alg\cC$), implying that $\Lambda(F,F)$ is finitely generated. Furthermore, by their universal properties, we have an isomorphism
$$\Lambda(F,G)\otimes \Lambda(F,G)\;\simeq\; \Lambda(F,G)\otimes \Lambda(F,F)$$
of $\Lambda(F,G)$-algebras (with the morphism from $\Lambda(F,G)$ coming from the identity of the left factors).
In conclusion, $\Lambda\otimes\Lambda$ is finitely generated as a $\Lambda$-algebra, for $\Lambda=\Lambda(F,G)$. By (faithfully flat) descent it follows that $\Lambda$ is itself indeed finitely generated. Concretely, for algebras $A,B$ such that $A\otimes B$ is finitely generated over $A$, write $B=\varinjlim B_\alpha$ for finitely generated subalgebras $B_\alpha\subset B$. It follows easily that one of the $A\otimes B_\alpha\subset A\otimes B$ must be an equality. As $A\otimes -$ is faithful and exact, we conclude $B_\alpha=B$.
\end{proof}

\begin{proof}[Proof of Theorem~\ref{ThmF1}]
We can reduce to the finitely generated case in Proposition~\ref{PropF1}, using Deligne's technique worked out in \cite[\S 6.4]{Tann}. Indeed, the method there is written out for the special case of $\bD=\sVec$, but works more generally in categories where certain assumptions on group schemes are satisfied, listed as \cite[6.3.1 - 6.3.5]{Tann}. That these are satisfied in our generality follows from \cite[7.2.6 and 7.2.7]{ComAlg} and Theorem~\ref{ThmNSub}, Lemma~\ref{LemAlgGrp} and Corollary~\ref{CorAlgGrp}.

As the result we need is a by-product of a more general result in \cite[\S 6.4]{Tann}, we sketch a proof here too. By transfinite induction, it is sufficient to prove the following claim. `Assume that a pretannakian category $\cC$ is finitely generated over a tensor subcategory $\cC_1\subset \cC$. Any two tensor functors $F_1,F_2:\cC\to\cD$ which become isomorphic when restricted to $\cC_1$ (say via $f:F_1|_{\cC_1}\Rightarrow F_2|_{\cC_1}$) are already isomorphic on $\cC$.'

To prove the claim we write $\cC_1=\cup_\alpha \cC_\alpha$ for all finitely generated tensor subcategories $\cC_\alpha\subset \cC$. We denote by $X\in\cC$ a generator of $\cC$ over $\cC_1$ and by $\cC_\alpha'\subset \cC$ the tensor subcategory generated by $\cC_\alpha$ and $X$. Clearly $\cC=\cup_\alpha\cC_\alpha'$. We denote by $\cC_X\subset\cC$ the tensor subcategory generated by $X$. There exists an isomorphism $g:F_1|_{\cC_X}\Rightarrow F_2|_{\cC_X}$ by \autoref{PropF1}. By \cite[7.2.6]{ComAlg} we can assume that $g$ and $f$ agree on $\cC_1\cap\cC_X$. Also by \autoref{PropF1} there exists an isomorphism $F_1|_{\cC_\alpha'}\Rightarrow F_2|_{\cC'_\alpha}$.  Subsequently, by \cite[7.2.6 and 7.2.7]{ComAlg} there actually exists a {\em unique} isomorphism $F_1|_{\cC_\alpha'}\Rightarrow F_2|_{\cC'_\alpha}$ which restricts to $f$ and $g$ on $\cC_\alpha$ and $\cC_X$. By uniqueness these isomorphisms can be glued together on $\cC=\cup_\alpha\cC'_\alpha$, showing $F_1\simeq F_2$.
\end{proof}

\subsection{Some technical tools}
In this section we derive some tools for verifying the maximally nilpotent condition. Recall from \cite{Tann}, the (additive) functor $\Fr_+:\cC\to\cC$, which sends an object $X$ to the image of the canonical morphism from the $S_p$-invariants $\Gamma^pX$ in $X^{\otimes p}$ to the co-invariants $\Sym^pX$.

\begin{lemma}\label{LemFr}
For a morphism $\alpha:\unit\to Y$, the following conditions are equivalent:
\begin{enumerate}
\item $\alpha^j=0$ for some $j\in\mN$;
\item $\Fr_+^i(\alpha^j)=0$ for some $i,j\in\mN$;
\item $(\Fr_+^i(\alpha))^j=0$ for some $i,j\in\mN$.
\end{enumerate}
\end{lemma}
\begin{proof}
That (1) implies (2) and (3) is trivial.

We will freely use that $(\alpha^a)^b=0$ implies $\alpha^{ab}=0$. Moreover, we can observe that $\alpha^p$ factors as 
\begin{equation}\label{ap}
\alpha^p:\;\,\unit\xrightarrow{\Fr_+(\alpha)} \Fr_+ Y\subset \Sym^p Y,
\end{equation}

In particular, this shows that $\Fr_+^i(\alpha)=0$ implies that $\Fr_{+}^{i-1}(\alpha^p)=0$. By iteration this demonstrates that (2) implies (1).

Another consequence of equation~\eqref{ap} is that $(\Fr_+\alpha)^n=0$ implies $(\alpha^{p})^n=0$. Therefore that (3) implies (1) also follows by iteration.
\end{proof}

\begin{lemma}\label{LemPavn}
Let $L$ be a non-invertible simple object in $\bT$ with a non-zero morphism $\phi:\otimes^n L\to S$ for $S\in\cC$ invertible.  For $E=\ker\phi$, we have $$E\otimes L+L\otimes E=L^{\otimes n+1}.$$
\end{lemma}
\begin{proof}
We need to prove that the composition
$$E\otimes L\;\hookrightarrow\;L^{\otimes n+1}\;\stackrel{L\otimes \phi}{\tto}\; L\otimes S$$
is an epimorphism, or equivalently, is not zero. Define
$$\psi:= (\phi\otimes L^\vee)\circ (L^{\otimes n-1}\otimes \co_L)\;:\; L^{\otimes n-1}\to S\otimes L^\vee,$$
which is non-zero and therefore an epimorphism.
By adjunction it is sufficient to prove that the kernel $K\subset L^{\otimes n}$ of $L\otimes \psi$ does not contain $E$. By postcomposing the latter morphism (and denoting the braiding by $\gamma$), we find
$$(\ev_L\otimes S)\circ \gamma_{L\otimes S,L^\vee}\circ (L\otimes \psi) \;=\;\phi\circ \gamma_{L,L^{\otimes n-1}}.$$ The displayed morphism has kernel $K'\subset L^{\otimes n}$, which {\em strictly} contains $K$ (as $L$ is not invertible) and by construction satisfies $L^{\otimes n}/K'\simeq S$. Clearly this prevents $E\subset K$, as we already have $L^{\otimes n}/E\simeq S$, which concludes the proof.
\end{proof}

\begin{corollary}\label{CorNInv}
Let $L$ be a non-invertible simple object in $\bT$ with a non-zero morphism $\phi:\Sym^n L\to S$ to an invertible $S\in\cC$. Assume that the kernel $K=\ker \phi$ satisfies either
\begin{enumerate}
\item $K$ is nilpotent as a degree 1 subobject in $\Sym (\Sym^n L)$, or;
\item $K$ is nilpotent as a degree n subobject in $\Sym(L)$; 
\end{enumerate}
 then the algebra $\Sym(L)$ is finite.
\end{corollary}
\begin{proof}
One can observe from the universal algebra morphism $\Sym(\Sym^nL)\to \Sym(L)$ that condition (1) actually implies condition (2). We thus prove the corollary assuming (2).

By Lemma~~\ref{LemPavn}, $\Sym^{n+1}L\subset\Sym(L)$ is contained in the ideal generated by $K$, so by nilpotency of $K$, we find $\Sym^mL=0$ for some multiple $m$ of $n+1$, and consequently $\Sym(L)$ is finite.
\end{proof}
We also state the following special case:
\begin{corollary}\label{Corn1}
Let $L$ be a non-invertible simple object in $\bT$ for which $\Sym^n L$ is invertible for some $n>1$, then $\Sym(L)$ is finite (in fact $\Sym^{n+1}L=0$).
\end{corollary}

We conclude this section with the following `converse' observation.

\begin{lemma}\label{LemInv}
For $X\in\cC$ with $\Sym(X)$ finite, $\Sym^d X$ is invertible, for the maximal $d\in\mN$ with $\Sym^dX\not=0$. Moreover, we have isomorphisms
$$\Sym^iX\;\xrightarrow{\sim}\; (\Sym^{d-i}X)^\vee\otimes\Sym^dX,\quad\mbox{for }\; 0\le i\le d.$$
\end{lemma}

\begin{proof}
Recall that an algebra $A$ (not necessarily commutative) in $\cC$ is Frobenius if for some $L\in\cC$ (automatically invertible), there is an isomorphism of left $A$-modules
$$A\;\xrightarrow{\sim}\; A^\vee\otimes L.$$

Since $A:=\Sym(X)$ is a finite Hopf algebra, it is a Frobenius algebra, see \cite[Corollary~4.5]{Ta}.
We restrict the displayed isomorphism to the generator 
$$\unit=\Sym^0 X\to (\Sym(X))^\vee\otimes L.$$ Then the restriction to $X=\Sym^1 X$ takes values in $(\Sym^{<d} X)^\vee\otimes L$, by the module structure on $(\Sym(X))^\vee$. By these observations for all degrees and the fact that $A\to A^\vee\otimes L$ is an epimorphism, we thus find epimorphisms
$$\Sym^iX\;\tto\; (\Sym^{d-i}X)^\vee\otimes L.$$
They must then all be isomorphisms, for instance by length consideration.
\end{proof}


\section{The Verlinde category}
In this section we set $\mathrm{char}(\bk)=p>0$.
In \cite{CEO}, it was proved that $\Ver_p$ is subterminal (see also Theorem~\ref{ThmF1}). In this section we demonstrate that it is also Bezrukavnikov.

\subsection{Growth dimension}\label{Secgd}
Fix a pretannakian category $\cC$ of moderate growth over a field of arbitrary characteristic.

\subsubsection{}Recall (see for instance \cite[\S 4]{CEO}) that for every $X\in\cC$ we have a well-defined limit
$$\gd(X)=\lim_{n\to\infty} d_n(X)^{1/n}\in\mR,\quad\mbox{with }\; d_n(X)=\ell(X^{\otimes n}).$$
Clearly $\gd(X)\ge 1$ if and only if $X\not=0$.
This yields a function, denoted by the same symbol,
$$\gd:\; K_0(\cC)\to\mR.$$
For $X,Y\in\cC$ and a short exact sequence $X\hookrightarrow Z\tto Y$ we have
$$\gd(X\otimes Y)\ge \gd(X)\gd(Y)\quad\mbox{and}\quad \gd(Z)\ge \gd(X)+\gd(Y).$$
On the other hand, $\gd(X^{\otimes n})=\gd(X)^n$ follows by definition. Consequently, $\gd$ is a ring homomorphism whenever it is a group homomorphism.
\begin{lemma}\label{LemQua}
Let $F:\cC\to\cC'$ be a tensor functor to a second pretannakian category $\cC'$.
 For any $X\in\cC$, we have 
$$\ell(FX)\le \gd(FX)\le \gd(X)^2.$$
\end{lemma}
\begin{proof}
It suffices to prove $\ell(FX)\le \gd(X)^2$, the sharper inequality $\gd(FX)\le \gd(X)^2$ is then obtained from applying the former to $X^{\otimes n}$.

Let $\phi_n: \bk S_n\to \End(X^{\otimes n})$ denote the algebra morphisms coming from the braiding. We have inequalities
$$\dim \im\phi_n \;\le\; \dim \End(X^{\otimes n})\;\le\; \ell(X^{\otimes n})^2.$$
It now suffices to prove 
$$\ell:=\ell(X)\le \limsup_{n\to\infty}(\dim \im \phi_n)^{1/n}.$$ 
Indeed, since $F$ is symmetric and faithful, the same bound for $\ell(FX)$ instead of $\ell(X)$ then automatically holds too.

For some filtration on $X$, we have ${\rm gr} X=Y_1\oplus...\oplus Y_\ell$ for simple $Y_i$. Denote by $A_n,A_n^0$ the images of 
$\bk S_{n\ell}$ in the endomorphism algebras of $X^{\otimes n\ell}$, ${\rm gr}X^{\otimes n\ell}$, respectively; so 
$\dim A_n\ge \dim A_n^0$. Let $T_n$ be a set of representatives in $S_{n\ell}$ for $S_{n\ell}/S_n^{\times \ell}$. It is easy to see (by looking at its action on $\bigotimes_{i=1}^n Y_i^{\otimes \ell}$) that the image of $T_n$ in $A_n^0$ is a linearly independent set. 
Thus, by Stirling's formula 
$$
(\dim A_n)^{\frac{1}{n\ell}}\ge (\dim A_n^0)^{\frac{1}{n\ell}}
\ge \left(\frac{(n\ell)!}{n!^\ell}\right)^{\frac{1}{n\ell}}\to \ell,\ n\to \infty,
$$
as claimed. 
\end{proof}
\begin{remark}
Lemma~\ref{LemQua} extends easily to the braided (rather than symmetric) analogue.
\end{remark}

\subsubsection{} For $X,Y\in\cC$ we introduce two values.
We let $C_n(X,Y)$ be the maximal length of $L_1\otimes L_2$ where $L_1$, resp. $L_2$, is a simple composition factor of $X^{\otimes i}$, resp. $Y^{\otimes j}$, with $i,j\le n$.
We let $D_n(X)$ denote the maximum of $\gd(L)$ for $L$ a simple composition factor of $X^{\otimes i}$ with $i\le n$. 
Then we set
$$C(X,Y):=\limsup_{n\to\infty}C_n(X,Y)^{1/n}\quad\mbox{and}\quad D(X)=\limsup_{n\to\infty}D_n(X)^{1/n}.$$

The following lemma is straightforward.
\begin{lemma} \label{LemC}
If $C(X,Y)=1$ for all $X,Y\in\cC$, then the function $\gd:K_0(\cC)\to\mR$ is a ring homomorphism.
\end{lemma}

\begin{lemma}\label{LemFin} Consider a tensor functor $F:\cC\to\cD$ to a pretannakian category in which every finitely generated tensor subcategory contains finitely many simple objects (up to isomorphism).
\begin{enumerate}
\item For every $X\in \cC$, we have $D(X)=1$.
\item The function $\gd:K_0(\cC)\to\mR$ is a ring homomorphism.
\end{enumerate}
\end{lemma}
\begin{proof}
Part (2) follows from \cite[Lemmata~8.3 and~8.5]{CEO}. In particular, \cite[Lemma~8.5]{CEO} states that $\gd(FX)=\gd(X)$. Therefore \cite[Proposition~11.1]{EOf} shows that, for 
$$B_n(X)\;:=\; \max_{i\le n}\gd(\Sym^i(X^\vee\otimes X))\;=\; \max_{i\le n}\gd(\Sym^i(FX^\vee\otimes FX)),$$
we have $\lim_{n\to\infty} B_n(X)^{1/n}=1$. By \cite[Lemma~8.2]{CEO} and part (2), we also have
$$D_n(X) \le D_n(X)^2\le B_n(X),$$
from which part (1) follows.
\end{proof}

In particular, \cite[Conjecture~1.4]{BEO} would imply:
\begin{conjecture}
The function $\gd:K_0(\cC)\to\mR$ is always a ring homomorphism (for every pretannakian category $\cC$ of moderate growth).
\end{conjecture}

\begin{prop}\label{PropSurj}
Assume that $D(X)=1$ for all $X\in \cC$, and consider a tensor functor $F:\cC\to\cC'$ to a pretannakian category $\cC'$.
\begin{enumerate}
\item For every $X\in\cC$, we have 
$$\gd F(X)\;=\; \gd X.$$
\item For every $X\in\cC$ and $n\in\mN$, we have 
$$D_n(FX)\;\le\; D_n( X).$$
\item Assume that $F$ is surjective. Then $D(Y)=1$ for all $Y\in\cC'$. Furthermore, if $\gd: K_0(\cC)\to\mR$ is a ring homomorphism, then so is $\gd: K_0(\cC')\to\mR$. 

\end{enumerate}
\end{prop}
\begin{proof}
Applying Lemma~\ref{LemQua} to every simple consituent in $X^{\otimes n}$, we find
$$d_n(FX)\;\le\; d_n(X) D_n(X)^2,$$
proving part (1).

For part (2), it follows easily that $D_n(FX)$ is bounded by the maximum of the values $\gd(FL)$, where $L$ runs over simple constituents of $X^{\otimes i}$ for $i\le n$. By part (1), we have $\gd(FL)=\gd(L)$, so the conclusion follows.

Now we prove part (3). If $Y$ is a subquotient of $FX$ then clearly $D(Y)\le D(FX)$, so $D(Y)=1$ follows from part (2).
By Lemma~\ref{LemC} (and because the property $C(A,B)=1$ implies the same for all subquotients of $A$ and $B$), it suffices to prove that 
$$C(FX,FY)=1,$$
for all $X,Y\in\cC$. Observe that the maximum length of tensor products of simples appearing in powers of $FX$ and $FY$ is bounded by the maximum length of objects $F(L_1\otimes L_2)$ with $L_1,L_2$ simples appearing in powers of $X$ and $Y$. For $L_1$, resp. $L_2$, a simple constituent of $X^{\otimes i}$, resp. $Y^{\otimes j}$, with $i,j\le n$, we have (by part (1) and assumption)
$$\ell(F(L_1\otimes L_2))\le \gd(F(L_1\otimes L_2))= \gd (L_1\otimes L_2) =\gd(L_1) \gd(L_2)\le D_n(X)D_n(Y).$$ 
Consequently, $C_n(FX,FY)\le D_n(X)D_n(Y)$, from which the conclusion follows.
\end{proof}

\subsection{Application to the Verlinde category}

Now we assume that $\mathrm{char}(\bk)=p>0$.
The main result of \cite{CEO} states that a pretannakian category admits a tensor functor to $\Ver_p$ if and only if it is of moderate growth and Frobenius exact. Here `Frobenius exact' means that the Frobenius functor
$$\Fr:\cC\to \cC\boxtimes\Ver_p,$$
see \cite{Tann, EOf}, is exact (and thus a tensor functor).

\begin{theorem}\label{ThmVerp}
 $\Ver_p$ is a Bezrukavnikov category. In other words, if $F:\cC\to\cC'$ is a surjective tensor functor, where $\cC$ is of moderate growth and Frobenius exact, then also $\cC'$ is of moderate growth and Frobenius exact.
\end{theorem}

The remainder of the section is devoted to the proof of the theorem. Firstly, we make the assumption that $\cC$ (and so also $\cC'$) contains $\Ver_p$ as a subcategory. This is no loss of generality as we can always replace $\cC$ by $\cC\boxtimes\Ver_p$. We will subsequently use the same notation $\Fr$ for the composite 
$$\cC\xrightarrow{\Fr}\cC\boxtimes\Ver_p\xrightarrow{\otimes}\cC.$$
Using that $\cC$ admits a tensor functor to $\Ver_p$, Lemma~\ref{LemFin} and Proposition~\ref{PropSurj} show that~$\gd$ is a ring homomorphism on $K_0(\cC)$ and $K_0(\cC')$, and we have
\begin{equation}\label{eqgdF}
\gd(FX)=\gd(X)=\gd(\Fr X),\quad\mbox{for all $X\in\cC$.}
\end{equation}

\begin{proof}[Proof of Theorem~\ref{ThmVerp}]
As $\cC$ is Frobenius exact, by equation~\eqref{eqgdF}, for all $X\in \cC$ we have 
\begin{equation}
\label{longeq}
\gd({\rm Fr}F(X))=\gd(F({\rm Fr}X))=\gd({\rm Fr}X)=\gd(X)=\gd(F(X)).
\end{equation}

Now let $Y\in \cC'$ and $X\in \cC$ be such that $Y$ is a subquotient of $F(X)$. Since ${\rm Fr}Y$ is a subquotient of $Y^{\otimes p}$, we 
have 
\begin{equation}\label{estFr}
\gd({\rm Fr}Y)\le \gd(Y^{\otimes p}) =\gd(Y)^p.
\end{equation}

In $K_0(\cC')$ we have 
$$
Y^{\otimes n}=\sum_Z N_n(Y,Z)Z,
$$
where the $Z$ are simple, so 
\begin{equation}\label{gdN}
\gd(Y)^n=\gd(Y^{\otimes n})\ge \sum_Z N_n(Y,Z).
\end{equation}
Since $Z$ is a simple constituent of $FX^{\otimes n}$, Proposition~\ref{PropSurj}(2) implies that
$$\gd(Z)\le D_n(FX)\le D_n(X).$$

Thus, by applying \eqref{estFr} to $Z$, we find that $\gd({\rm Fr}Z)\le D_n(X)^p$. This, combined with Lemma \ref{esti}(1) below and equation \eqref{gdN}, implies
$$
\gd({\rm Fr}Y)^n=\gd({\rm Fr}Y^{\otimes n})\le \sum_Z N_n(Y,Z)\gd({\rm Fr}Z)\le 
(\sum_Z N_n(Y,Z))D_n(X)^p\le \gd(Y)^n  D_n(X)^p.
$$
It follows that 
$$
\gd({\rm Fr}Y)\le \gd(Y)D_n(X)^{\frac{p}{n}}.
$$
Thus Lemma \ref{LemFin}(1) implies that 
$$
\gd({\rm Fr}Y)\le \gd(Y).
$$
Consider the filtration on $F(X)$ with successive quotients $T_1,Y,T_2$. 
By Lemma \ref{esti}(1) below and equation~\eqref{longeq} we thus have 
$$
\gd(F(X))=\gd({\rm Fr}F(X))\le \gd({\rm Fr}T_1)+\gd({\rm Fr}Y)+\gd({\rm Fr}T_1)
$$
$$
\le 
\gd(T_1)+\gd(Y)+\gd(T_2)=\gd(F(X)). 
$$
Thus all inequalities in this chain are equalities, so $\gd({\rm Fr}Y)=\gd(Y)$. Lemma~\ref{esti}(2) below then shows that $\cC_1$ is Frobenius exact. 
\end{proof}

\begin{lemma}\label{esti}
Let $\cD$ be a pretannakian category of moderate growth with $D(X)=1$ for all $X\in\cD$ for which $\gd: K_0(\cD)\to\mR$ is a ring homomorphism. Assume $\Ver_p\subset\cD$, so that we can consider $\Fr:\cD\to\cD$.

\begin{enumerate}
\item  If $0\to Y_1\to Y_2\to Y_3\to 0$ is a short exact sequence in $\cD$
then 
$$
\gd({\rm Fr}Y_2)\le \gd({\rm FrY}_1)+\gd({\rm Fr}Y_3).
$$
\item The following conditions are equivalent:
\begin{enumerate}
\item $\Fr$ is exact on $\cD$;
\item $\gd(\Fr X)=\gd X$ for all $X\in\cD$;
\item The inequality of part (1) is always an equality.
\end{enumerate}
\end{enumerate}

\end{lemma} 

\begin{proof} 
Both parts are proved in \cite{EOf} for finite tensor categories (in which case the condition on $\gd=\FPdim$ and $D(-)$ are automatically satisfied. The same methods apply to our setting. Concretely, just like \cite[Proposition~6.1]{EOf}, part (1) follows from \cite[3.10 and 3.11]{EOf}. That (2)(a) implies (2)(b) is an example of Proposition~\ref{PropSurj}(1). That (2)(b) implies (2)(c) is obvious. That (2)(c) implies (2)(a) follows from the proof of \cite[6.4 (v) $\Rightarrow$ (vi)$\Rightarrow$(iii)]{EOf}, which one can copy verbatim.
\end{proof} 

\begin{remark}
As can be verified directly from the proof, part (1) and the implications (b)$\Rightarrow$(c)$\Rightarrow$(a) remain valid without the condition $D(X)=1$.
\end{remark}



\section{The higher Verlinde categories}

In this section, we investigate whether $\Ver_{p^\infty}$ is subterminal. By Theorem~\ref{ThmF1} we can focus on verifying whether the categories $\Ver_{p^n}$ are {\bf MN} and {\bf GR}. In \cite{ComAlg} it was already proved that $\Ver_{p^2}$ satisfies is {\bf GR} and here we mainly focus on {\bf MN}. This is also motivated by Theorem~\ref{LemNoSim}(2), which begs the question of whether every (finite) incompressible tensor category is {\bf MN}.

By \autoref{Cond4Fib}(3), a category is maximally nilpotent if and only if the following two conditions are satisfied:
\begin{enumerate}
\item[{\bf MNa}] For every non-trivial simple $L\in\bT$, the algebra $\Sym(L)$ is finite, and
\item[{\bf MNb}] For every simple $L\in\bT$ and indecomposable extension
$$ 0\to \unit^m\to X\to L\to 0,$$
the image of $\Sym(\unit^m)\to \Sym(X)$ is finite.
\end{enumerate}

For the remainder of the section we assume $\mathrm{char}(\bk)=p>0$.

\subsection{{\bf MNb} is always satisfied}

\begin{prop}\label{Proppb}
The category $\Ver_{p^n}$ satisfies {\bf MNb}.
\end{prop}

First we need the following technical result.
\begin{lemma}\label{Lem2p2}
Assume $p>2$. For $L_{2p-2}\in \Ver_{p^2}$, the object $\unit$ is a direct summand of $\Sym^2L_{2p-2}$.
\end{lemma}
\begin{proof}
We have
$L_{2p-2}\simeq L_p\otimes L_{p-2}.$
By the rules for $\Ver_p\subset \Ver_{p^2}$, we have $\Sym^2L_p\simeq L_{2p}$ and $\wedge^2 L_p\simeq\unit$. Hence, $\Sym^2 L_{2p-2}$ has a direct summand
$\wedge^2 L_{p-2}$. Since the defining monoidal functor satisfies
$$\Tilt SL_2\;\to\; \Ver_{p^2},\quad T_{p-2}\mapsto L_{p-2},$$
it suffices to observe that $T_0$ is a direct summand in $\wedge^2 T_{p-2}$. This follows for instance because
$\Tilt SL_2\to \Ver_{p}$ sends $T_{p-2}$ to the odd line of $\sVec\subset \Ver_p$.
\end{proof}
\begin{proof}[Proof of Proposition~\ref{Proppb}]
We start by considering the case $p>2$.
From the description in \cite[Proposition~4.65]{BEO} we find
$$
\dim_k\Ext^1_{\Ver_{p^n}}(L_a,\unit)=\begin{cases}
1&\mbox{ if $a=p^i+p^{i-1}(p-2)$ for some $1\le i\le n-1$;}\\
0&\mbox{ otherwise,}
\end{cases}
$$
see \cite[Lemma~4.2.3]{CEH}. So it suffices to consider $m=1$ in {\bf MNb}. Denote by $E_i^{n}\in \Ver_{p^n}$ the object in $\Ver_{p^n}$ resulting from the above extension for $a=p^i+p^{i-1}(p-2)$. 

We have $E_i^n\in \Ver_{p^{n-1}}\subset \Ver_{p^n}$ for $i>1$, which thus implies $E_i^n\simeq E_{i-1}^{n-1}$. Furthermore, by \cite[Lemma~4.4.6]{CEH}, we have 
$$\Fr_+ E^n_i\;\simeq\; E^{n-1}_i,\qquad\mbox{if } i<n-1.$$
By Lemma~\ref{LemFr}, it is therefore sufficient to consider $E_{n-1}^n$ for all $n$. However, by the above, we have $E^n_{n-1}\simeq E^2_1=:E_1\in \Ver_{p^2}$.

We thus consider the non-split short exact sequence
$$0\to \unit\xrightarrow{\alpha} E_1\to L_{2p-2}\to 0$$
in $\Ver_{p^2}$. We will prove that $(\alpha^2)^i=0$ for some $i$ (so in particular $\alpha^{2i}=0$).
Since $p>2$, the morphism $\alpha^2$ is a composite
$$\unit\subset E_1\subset \Sym^2 E_1.$$
Let $P$ be the projective cover of $\unit$. It is also the injective hull of $\unit$ and has length three with $L_{2p-2}$ in the middle, see \cite[\S 5.1]{BEO}. We claim the induced $\Sym^2E_1\to P$ is an epimorphism, and therefore split (in fact one can easily show it must even be an isomorphism). This allows us to consider $\beta:\unit\hookrightarrow P$ instead of $\alpha^2$. 

To prove this claim, we can observe that $[\Sym^2 E_1:\unit]>1$ by Lemma~\ref{Lem2p2}. Furthermore, by adjunction we find $\Hom(\unit, E_1^{\otimes 2})$ has dimension one, so the same is true for $\Hom(\unit, \Sym^2E_1)$, as $\Sym^2 E_1$ is a direct summand of $E_1^{\otimes 2}$.
Moreover, one can calculate from \cite[Corollary~4.50]{BEO} that $[L_{2p-2}^{\otimes 2}:L_{2p-2}]=0$ and therefore $[\Sym^2 E_1:L_{2p-2}]=1$, where $L_{2p-2}$ is the only simple object in $\Ver_{p^2}$ which extends with~$\unit$. The claim now follows.

We can thus turn to $\beta:\unit\hookrightarrow P$. We have $\Fr P=0$, see \cite{BEO}. This also follows for instance because $\Fr$ is monoidal but not faithful (on $\Ver_{p^2}$) and projective objects are contained in any non-zero tensor ideal. In particular, we have $\Fr_+ P=0$, since $\Fr_+P$ is always a subquotient of $\Fr P$.
The conclusion now follows from Lemma~\ref{LemFr}.

Now we focus on $p=2$. Extensions between simple objects are described in \cite[Proposition~4.12]{BE}. Similarly to the case $p>2$, we can reduce to short exact sequences (unique up to scalar)
$$0\to \unit\to E\to \unit\to 0$$
in ${\Ver_4}^+=\mathcal{C}_1$ and 
$$0\to \unit \to E_1\to V\to 0$$
in $\Ver_8^+=\mathcal{C}_3$, with $V$ the non-trivial simple object $\Ver_8^+$ (which already lives in $\Ver_4=\cC_2$). The first case is easy since, after applying the non-symmetric tensor functor $\Ver_4^+\to \Vecc$, the relevant algebra morphism $\Sym\unit\to \Sym E$ becomes, see \cite[5.2.1]{BE},
$$  \bk[y]\to \bk[x,y]/y^2,\quad y\mapsto y.$$

For the second case, we claim that
$$E_1\otimes E_1\;\simeq\; P/\unit,$$
with $P$ the projective cover and injective hull of $\unit$. Indeed, if we denote by $P_V$ the projective cover of $V$, see \cite[\S 5.3]{BE}, then we have a short exact sequence
$$0\to E_1\otimes E_1^\vee\to E_1\otimes P_V\to E_1\otimes E_1\to 0.$$
One can combinatorially verify that the middle term is $P_V\oplus P$. The fusion rules, the fact that the categorical dimension of $E_1$ is not zero and the observation (by adjunction) that
$$\Hom(E_1\otimes E_1^\vee,\unit)\;=\;\bk\;=\; \Hom(\unit, E_1\otimes E_1^\vee),$$
show that $E_1\otimes E_1^\vee$ is a direct sum of $\unit$ and $P_V$, implying the claim.

We know that $\Sym^2 E_1$ is the cokernel of a (non-zero) morphism $E_1^{\otimes 2}\to E_1^{\otimes 2}$. This leads to two possibilities: either $\unit\to \Sym^2E_1$ is zero (this option can actually easily be verified not to be true, but we do not need that), or we have
$$\Sym^2E_1\simeq V_2^{\otimes 2}$$
with $V_2\in \cC_4$ as in \cite[\S 5.4]{BE}. Since $\Fr_+^2V_2=0$ (see \cite[Theorem~2.1(xi)]{BE}) and $\Fr_+$ is monoidal, the conclusion follows from Lemma~\ref{LemFr}.
\end{proof}

%

\subsection{Partial results on {\bf MNa}}

\begin{theorem}\label{PropVer2}
$\Ver_{2^\infty}$ is maximally nilpotent.
\end{theorem}
\begin{proof}
By Proposition~\ref{Proppb}, it is sufficient to prove {\bf MNa} for each $\cC_{2n-2}=\Ver_{2^n}$.

We prove, by induction on $n$, that the condition is satisfied for every simple object in $\Ver_{2^n}$ (which has the same simple objects as $\cC_{2n-1}=\Ver_{2^{n+1}}^+$). Firstly, by \cite[Theorem~2.1]{BE}, for such a simple $L\in\Ver_{2^n}$, the power $L^{\otimes 2}$ belongs to $\Ver_{2^{n}}^+$. By the induction hypothesis and Proposition~\ref{Proppb}, $\Ver_{2^n}^+$ is {\bf MN}.

We have $\Hom(\Sym^2 L,\unit)=\fk$ (in any case, the alternative $\Hom(\Sym^2 L,\unit)=0$ would conclude the proof immediately). By Corollary~\ref{CorNInv}(1), it is sufficient to show that~$K$ (the kernel of $\Sym^2L\tto\unit$) is nilpotent in $\Sym(\Sym^2 L)$, which we do now.

In case $\Hom(K,\unit)=0$, we know $\Sym K$ is finite, by the first paragraph, so certainly~$K$ is nilpotent in $\Sym(\Sym^2 L)$. Since $\Ext^1(\unit,\unit)=\fk$, see \cite[\S 4.1]{BE}, the other option is $\Hom(K,\unit)=\fk$, when we have a short exact sequence
$$0\to K'\to \Sym^2 L\to E\to 0,$$
where $E$ is the self-extension of $\unit$ (the indecomposable projective object in $\Ver_4^+$), and moreover $\Hom(K',\unit)=0$. The latter implies again that $\Sym K'$ is finite. Since the composite $\unit \hookrightarrow E^{\otimes 2}\tto \Sym^2 E$ is zero (see the proof of \autoref{Proppb}), it follows that the subobject $K^2\subset \Sym^2 (\Sym^2L)$ is in the ideal of $\Sym(\Sym^2 L)$ generated by $K'$. As the latter ideal is nilpotent ($\Sym K'$ is finite), it follows that indeed $K$ is nilpotent. 
\end{proof}


\begin{lemma}\label{Ver9}
$\Ver_9$ is maximally nilpotent.
\end{lemma}
\begin{proof}
Recall the fundamental properties from~\ref{DefsVer}. In particular the simple objects in $\Ver_9$ are labelled by $L_i$, $0\le i \le 5$.
We will also use that $L_3$  is the odd line (in particular $\Sym^2L_3=0$) and that $L_0, L_1, L_2$ are the images of the correspondingly labelled $SL_2$-tilting modules. Since $L_4\simeq L_3\otimes L_1$, we have
$$\Sym^2 L_4\simeq \wedge^2L_1\simeq \unit.$$
Consequently, $\Sym^3L_4=0$ by Corollary~\ref{Corn1}.

Since $\Sym^2 L_2$ is the image of $\Sym^2T_2\simeq T_4$, it is the projective cover of $\unit$, see \cite[\S 5.1]{BEO}. Its radical $X$ is of length 2, with top $L_4$ and socle $\unit$. By Proposition~\ref{Proppb}, the ideal in $\Sym(X)$ generated by $\unit$ in degree 1 is nilpotent, and by the previous paragraph, the quotient with respect to the ideal is finite. Hence $\Sym(X)$ is a finite algebra. That $\Sym(L)_2$ is finite now follows from Corollary~\ref{CorNInv}. 

Finally, we can observe, using the defining monoidal functor, resp. $L_5\simeq L_3\otimes L_2$, that
$$\Sym^2 L_1\;\simeq\; L_2\;\simeq\;\wedge^2 L_2\;\simeq\; \Sym^2 L_5 .$$
That $\Sym(L)_1$ and $\Sym(L)_5$ are finite thus follows from the fact that $\Sym(L)_2$ is finite.
\end{proof}

Using similar methods as in the proof of Lemma~\ref{Ver9}, we can prove the following lemma.

\begin{lemma}
$\Ver_{27}^+$ satisfies is {\bf MN} if and only if $\Sym(L)_8$ is finite.
\end{lemma}

Part (1) of the following lemma was also observed in \cite[Corollary~6.4]{BE2}.

\begin{lemma}Consider $\Ver_{p^n}$.
\begin{enumerate}
\item For $0\le i<n$, we have
$$\Sym^{p^{n-i}-1}L_{p^i}=0.$$
\item For $0\le j<n-1$ and $0\le i<p$, we have
$$\Sym^{i+2}L_{p^{n-1}(p-2)+p^ji}=0.$$
\item For $0< i<p-1$, we have
$$\Sym^{p-i}L_{p^{n-1}i}=0.$$
\end{enumerate}
\end{lemma}
\begin{proof}
For part (1), by letting $n$ vary, it suffices to prove that
$$\Sym^{p^{n}-1}L_1=0.$$
For this we can observe that in $\Rep SL_2$, the $p^{n}-1$-th symmetric tensor power of $T_1$ is the Steinberg module $T_{p^{n}-1}$. By highest weight considerations and the block decomposition, it follows that the exact sequence
$$(T_1^{\otimes p^{n}-1})^{p^n-2}\;\to\;T_1^{\otimes p^{n}-1}\;\to\; T_{p^n-1}\;\to\; 0$$
 is split. Since $T_{p^n-1}$ is in the ideal of the defining monoidal functor to $\Ver_{p^n}$ (it is in fact the canonical generator of this ideal), the conclusion follows from \cite[Proposition~3.5]{BEO}.

For part (2), it suffices to consider the case $j=0$. So for $n>1$ we need to show that
$$\Sym^{i+2}L_{p^{n-1}(p-2)+i}=0.$$
Since
$$L_{p^{n-1}(p-2)+i}\;\simeq\; L_{p^{n-1}(p-2)}\otimes L_i$$
and $L_{p^{n-1}(p-2)}$ is the odd line in $\sVec$, the statement reduces to $\wedge^{i+2}L_i=0$. The latter is true if $i+2<p$ by $\wedge^{i+2}T_i=0$. For the case $i=p-2$, we can use a similar reasoning: Since $\wedge^{p-1} T_{p-2}=\bk$, and hence $\Sym^{p-1}L_{p^{n-1}(p-2)+p-2}=\unit$, the conclusion follows from Corollary~\ref{Corn1}. Finally, consider $i=p-1$. In this case $\wedge^pT_{p-1}=\bk$, and moreover the highest weight appearing in $T_{p-1}^{\otimes p}$ is $p^2-p$, below $p^2-1$. Hence none of the tilting objects in the canonical presentation of $\wedge^p T_{p-1}$ are in the kernel of the defining monoidal functor. It follows again via \cite[Proposition~3.5]{BEO} that $\wedge^p L_{p-1}=\unit$ and we can again conclude via Corollary~\ref{Corn1}.

Part (3) follows from a simpler but analogous reasoning as in (2).
\end{proof}

\subsection{Algebras of finite type}


\begin{definition}
An algebra $A\in\Alg\bT$ is of {\bf finite type} if one of the following equivalent conditions is satisfied:
\begin{enumerate}
\item The $\bk$-algebra $A^{\inv}:=\Hom(\unit, A)$ is finitely generated and $A$ is a finitely generated $A^{\inv}$-module;
\item There exists a finitely generated $\bk$-algebra $R\subset A^{\inv}$ over which $A$ is a finitely generated module.
\end{enumerate}
\end{definition}

\begin{remark}
It is obvious that condition (1) implies (2). On the other hand, it is easy to see that a finitely generated $\bk$-algebra remains noetherian as an algebra in $\Ind\bT$. From this it follows indeed that (2) also implies (1).
\end{remark}

We omit the short proof of the following lemma.
\begin{lemma}\label{LemNoe}
\begin{enumerate}
\item Algebras of finite type are finitely generated and noetherian.
\item A quotient of an algebra of finite type is again of finite type.
\end{enumerate}
\end{lemma}

\begin{corollary}\label{CorFinTyp}
\begin{enumerate}
\item If every finitely generated algebra in $\Alg\bT$ is of finite type, then~$\bT$ (and every pretannakian category with a tensor functor to $\bT$) satisfies the Hilbert Basis Property.
\item If every finitely generated algebra in $\bT$ is of finite type, then $\bT$ is {\bf GR}.
\end{enumerate}
\end{corollary}
\begin{proof}
The first part follows from Lemma~\ref{LemNoe}(1). The second part follows directly from the definition and \cite[Theorem~3.1.5(1)]{ComAlg}.
\end{proof}

\begin{conjecture}\label{ConjFin}
If $\mathrm{char}(\bk)>0$, then every finitely generated algebra in a finite tensor category over $\bk$ is of finite type.
\end{conjecture}

\begin{remark}
\begin{enumerate}
\item Conjecture~\ref{ConjFin} is false in characteristic zero, see \cite[Remark~2]{Ve}.
\item Conjecture~\ref{ConjFin} is true (for all characteristics in fact) for finite tannakian categories, by the Hilbert-Noether theorem.
\item Conjecture~\ref{ConjFin} would in particular provide an affirmative answer to the question in \cite[Remark~3.1.7(3)]{ComAlg}.
\end{enumerate}
\end{remark}

The following theorem extends a result on $\Ver_{4}^+$ from \cite[Theorem~1.8]{Ve}.

\begin{theorem}\label{ThmVer22}
Every finitely generated algebra in $\Ver_{2^\infty}$ is of finite type. In particular, $\Ver_{2^\infty}$ is geometrically reductive.
\end{theorem}
\begin{proof}
It suffices to prove this for every category $\Ver_{2^n}=\cC_{2n-2}$. We will prove this by induction on $n$ (with $\Ver_{2}=\Vecc=\cC_0$ being trivial) and will use $\Ver_{2^n}^+=\cC_{2n-3}$. Assume the property is satisfied for $\Ver_{2^{n-1}}$, and we will prove it for $\Ver_{2^n}$.

Every finitely generated algebra in $\Ind\cC_{2n}$ is a quotient of an algebra $\Sym P$ for a projective $P\in\cC_{2n}$. Furthermore, every projective object appears as a direct summand in a tensor power of the generating object $X_n$, see for instance the final paragraphs of \cite[\S 3.8]{BE}. By Lemma~\ref{LemNoe}(2), it thus suffices to prove that $B_m:=\Sym(X_n^{\otimes m})$ is of finite type for $m\in\mN$.

Consider the even degree subalgebra $B_{m+}$ inside $B_m$. Since $B_m$ is a finite module over $B_{m+}$, it suffices to show that $B_{m+}$ is of finite type. But $B_{m+}$ is a quotient of $B_{2m}$, so it suffices to show that $B_{2m}=\Sym(X_n^{\otimes 2m})$ is of finite type. Now $X_n^{\otimes 2m}$ belongs to $\cC_{2n-1}$, see \cite[Theorem~2.1(vi) and (ix)]{BE}.

Recall that $\cC_{2n-1}$ can be realized as the category of modules in $\cC_{2n-2}$ over a certain triangular Hopf algebra $H=(\wedge X_{n-1},R)$, 
where 
$$
R=1\otimes 1+\tau+a\otimes a, 
$$
where $\tau$ is a generator of the invariants in $X_{n-1}\otimes X_{n-1}\subset H\otimes H$ 
and $a$ is a generator of the socle of $H$. Moreover, $X_n^{\otimes 2m}=(X_n\otimes X_n)^{\otimes m}=H^{\otimes m}$ is a graded $H$-module. 
So it suffices to prove that $A:=\Sym(Y)$ is of finite type, with
$Y\in \cC_{2n-1}$ a graded $H$-module. So $A$ may be regarded as a (possibly non-commutative) algebra in  $\cC_{2n-2}$
with an $H$-action. When we view it as such, we denote it by $B$ (formally, $B$ is the image of $A$ under the non-symmetric tensor functor $\cC_{2n-1}\to\cC_{2n-2}$).

Now observe that the module $Y$ has a decreasing filtration by $H$-submodules $Y_{\ge n}:=\oplus_{i\ge n}Y[i]$, and ${\rm gr}Y$ is a trivial $H$-module which coincides with $Y$ as a graded object of $\cC_{2n-2}$. This induces a filtration on $A=\Sym(Y)$ compatible with the grading, such that ${\rm gr}A$ is in $\cC_{2n-2}\subset\cC_{2n-1}$ (or the ind-completion) as a quotient of $\Sym({\rm gr}Y)$. 
By the induction assumption, ${\rm gr}A$
is a finite module over its finitely generated subalgebra of invariants $({\rm gr}A)^{\rm inv}$. 
Let $x_1,...,x_n$ be homogeneous (in both gradings) generators of  $({\rm gr}A)^{\rm inv}$. 

Since $\gr A=\gr B\simeq B$ (the latter isomorphism not as algebras), these generators may also be viewed as generators of $B^{\rm inv}=\Hom(\unit, B)$, the subalgebra of invariants in $B$ in the category $\cC_{2n-1}$, and $B$ is a finite module over ${B^{\rm inv}}$, as ${\rm gr}({B^{\rm inv}})=({\rm gr}B)^{\rm inv}$. We are not done yet, though, as the elements $x_i$ in ${B^{\rm inv}}$ may not be in $A^{\inv}$, since
$$\Hom_{\cC_{2n-1}}(\unit, A)\;\subset\;\Hom_{\cC_{2n-2}}(\unit, B).$$

However, we can use the $x_i$ to construct elements which do belong to $A^{\inv}$. Namely, since $H$ is generated by a primitive object which therefore acts on $A=\Sym(Y)$ by derivations, and since $a(x)$ is central for all $x\in {B^{\rm inv}}$, it is easy to show that the elements\footnote{If $n=2$ then $\tau=0$ and for any $x\in {B^{\rm inv}}$ one has $a(x)^2=0$. Thus $y_i^2=x_i^4$. These are the elements used in \cite{Ve}. Moreover, one can show that 
for general $n$, $a(x)^{2^N}=0$ for some $N$ (it seems $N=n-1$ should work), so instead of $y_i$ we may use $y_i^{2^N}=x_i^{2^{N+1}}$.} $y_i:=x_i^2-a(x_i)x_i$ 
are $H$-stable, i.e., belong to $B^{\rm inv}$. 
Let $A_0$ be the subalgebra of $A^{\rm inv}$ generated by the $y_i$. 
Then ${B^{\rm inv}}$ is a finite module over $A_0$, as 
${\rm gr}(A_0)$ contains the elements $x_i^2$ of the commutative algebra $({\rm gr}A)^{\rm inv}$ (the leading parts of $y_i$), and this algebra is generated by the $x_i$. Hence $A$ is also a finite module over $A_0$. 
This completes the induction step. 
\end{proof}

\begin{remark}
By Lemma~\ref{LemNoe}, the conclusion of Theorem~\ref{ThmVer22} yields another proof that $\Ver_{2^\infty}$ satisfies the Hilbert Basis Property (see Theorems~\ref{LemNoSim}(5) and~\ref{PropVer2}).
\end{remark}

\subsection{Higher Verlinde categories are subterminal in characteristic 2}

\begin{theorem}\label{Thm2in}
The category $\Ver_{2^\infty}$ is subterminal. Concretely, if $\mathrm{char}(\bk)=2$, then for an arbitrary tensor category $\bT$ over $\bk$, there is at most one tensor functor (up to isomorphism)
$$\bT\;\to\; \Ver_{2^\infty}.$$
\end{theorem}
\begin{proof}
By \autoref{PropVer2} and Theorem~\ref{ThmVer22}, $\Ver_{2^\infty}$ is {\bf MN} and {\bf GR}. The conclusion thus follows from Theorem~\ref{ThmF1}.
\end{proof}

Theorem~\ref{Thm2in} implies that \cite[Question~1.2]{BE} is equivalent to the question of whether $\Ver_{2^\infty}$ is the terminal object in $\MdGr_{\bk}$ when $\mathrm{char}(\bk)=2$.

\section{Some consequences for $p$-adic dimensions}\label{SecEHO}

Let $\cC$ be a pretannakian category over a field $\bk$ of characteristic $p>0$.

\subsection{Questions}

\subsubsection{} For $X\in \cC$, we consider the Hilbert series of its symmetric algebra:
$$\HS_X(t):=\sum_{i\ge 0}\dim(\Sym^i X)\;\in\; \mF_p[[t]],$$
where we use that categorical dimensions take values in $\mF_p\subset\bk$, see \cite[Lemma~2.2]{EHO}.
By \cite[Theorem~2.3]{EHO}, there exists a $p$-adic integer $\Dim_+X\in\mZ_p$ such that
\begin{equation}\label{EqEHO}
\HS_X(t)\;=\; (1-t)^{-\Dim_+X},
\end{equation}
where for $d\in\mZ_p$ we set
$$(1-t)^{d}\;:=\; \prod_{j\ge 0}(1-t^{p^j})^{d_j}\in\mF_p[[t]],\quad\mbox{for }\; d=\sum_{j\ge 0}d_jp^j\quad\mbox{ (with $0\le d_j<p$)}.$$

\subsubsection{} In \cite[Question~3.13.2]{EHO} it is asked whether $\Dim_+(X)$ always equals $\Dim_+(X^\vee)$. In this section we show this is not the case. This also answers \cite[Question~3.13.1]{EHO}, of whether $\Sym^i X$ and $\Gamma^i X$ always produce the same element of the Grothendieck group~$K_0(\cC)$, in the negative.

Since we will obtain counterexamples for $p=3$, this also shows that the corresponding questions about exterior powers have a negative answer, by consider $\cC\boxtimes\sVec$.

\subsection{Negative answers}

The following is a general method to produce counterexamples (as either $V$ or $E$ must be a counterexample).
\begin{prop}\label{Proppadic}

Consider a short exact sequence
$$0\to \unit \to E\to V\to 0$$
in $\cC$ such that 
\begin{enumerate}
\item The image of $\Sym\unit\to \Sym E$ has finite length, say, $n\in\mZ_{>0}$;
\item The algebra $\Sym(V)$ is finite;
\item The morphism $\Sym (V^\vee)\to \Sym (E^\vee)$ is a monomorphism.
\end{enumerate} 
Then we have
$$\Dim_+(E)=\Dim_+(V)+1-n\qquad\mbox{and}\qquad \Dim_+(E^\vee)=\Dim_+(V^\vee)+1.$$
\end{prop}

\begin{remark}
As observed in \cite[\S 3.2]{ComAlg}, the element $n$ from Proposition~\ref{Proppadic} must be a power of $p$.
\end{remark}

We start with the following lemma.
\begin{lemma}\label{LemDim}
Assume that $\Sym(X)$ is a finite algebra, then
$$\Dim_+(X)\;=\; -\max\{i\in\mN\,|\, \Sym^iX\not=0\}\in\mZ\subset\mZ_p.$$
If $X\not=0$, then $\dim(\Sym(X))=0$.
\end{lemma}
\begin{proof}
By Lemma~\ref{LemInv}, we have
$$\HS_X(t)\;=\; 1 + a_1 t+\cdots + a_{d-1}t^{d-1}\pm t^d.$$
Equation~\eqref{EqEHO} shows that we must have $d=-\Dim_+(X)$, so in particular
$$\HS_X(t)\;=\; (1-t)^d.$$
As $\dim (\Sym(X))=\HS_X(1)$, this concludes the proof.
\end{proof}

\begin{example}(\cite[Remark~6.5]{BE2})
In $\Ver_{p^n}$ we have
$$\Dim_+ L_{1}\;=\; 2-p^n.$$
\end{example}

\begin{proof}[Proof of Proposition~\ref{Proppadic}]
The second equality follows immediately from Theorem~\ref{ThmSym}. That theorem also implies that
$$\gr \Sym E\;\simeq\;\bk[x]/x^{n}\otimes \Sym(V).$$
Hence $\Sym E$ is a finite algebra and the conclusion follows from applying Lemma~\ref{LemDim} to $E$ and $V$.
\end{proof}

We use the proposition to obtain explicit counterexamples to \cite[Question~3.13]{EHO} for $p\in\{2,3\}$, expecting that similar examples occur for $p>3$.
\begin{example}
 Consider the extension
$$0\to \unit \to E_1\to V\to 0$$
in $\Ver_8$ or $\Ver_9$ from the proof of Proposition~\ref{Proppb} (with $V^\vee\simeq V$ simple). By Theorem~\ref{PropVer2} and Lemma~\ref{Ver9}, conditions~\ref{Proppadic}(1) and~\ref{Proppadic}(2) are satisfied. Concretely, it is shown that $\Sym^2V=\unit$ (so $\Sym^3 V=0)$.
Condition \ref{Proppadic}(3) is thus equivalent to injectivity of
$$\unit=\Sym^2V\;\to\;\Sym^2 E^\vee_1.$$
This is automatically satisfied for the $p=3$ case ({\it i.e.} $\Ver_9$), since $2<3$. For the $\Ver_8$ example one can calculate $\Sym^2 E_2^\vee$ directly, similarly to the case $\Sym^2E_1$ in the proof of Proposition~\ref{Proppb}.

Moreover, one can show that for the $p=2$ case we have
$$\Dim_+ E_1 =-5\qquad\mbox{and}\qquad \Dim_+ E_1^\vee = -1.$$
\end{example}


\begin{thebibliography}
	{DMNO}

\bibitem[BE1]{BE}  D.~Benson, P.~Etingof: Symmetric tensor categories in characteristic 2. Adv. Math. 351 (2019), 967--999.

\bibitem[BE2]{BE2} D.~Benson, P.~Etingof:
On cohomology in symmetric tensor categories in prime characteristic. arXiv:2008.13149.

\bibitem[BEO]{BEO} D.~Benson, P.~Etingof, V.~Ostrik: New incompressible symmetric tensor categories in positive characteristic. Duke Math. J. 172 (2023), no. 1, 105--200.

\bibitem[Be]{B} R. Bezrukavnikov. On tensor categories attached to cells in affine Weyl groups. Representation theory of algebraic groups and quantum groups,
{\em Adv. Stud. Pure Math.} {\bf 40}, Math. Soc. Japan, Tokyo, 2004. 

\bibitem[CO]{CO} J.~Comes, V.~Ostrik:
On Deligne's category $\underline{\rm Re}{\rm p}^{ab}(S_d)$.
Algebra Number Theory 8 (2014), no. 2, 473--496. 


\bibitem[Co1]{Tann} K.~Coulembier: Tannakian categories in positive characteristic. Duke Math. J. 169 (2020), no. 16, 3167--3219.

\bibitem[Co2]{AbEnv}  K.~Coulembier: Monoidal abelian envelopes. Compos. Math. 157 (2021), no. 7, 1584--1609.

\bibitem[Co3]{ComAlg} K.~Coulembier: Commutative algebra in tensor categories. In preparation.

\bibitem[CEH]{CEH} K.~Coulembier, I.~Entova-Aizenbud, T.~Heidersdorf: Monoidal abelian envelopes and a conjecture of Benson and Etingof. Algebra Number Theory 16 (2022), no. 9, 2099--2117.

\bibitem[CEO]{CEO} K.~Coulembier, P.~Etingof, V.~Ostrik: On Frobenius exact symmetric tensor categories. With an appendix by A.~Kleshchev. Ann. of Math. (2) 197 (2023), no. 3, 1235--1279.


\bibitem[CEOP]{CEOP} K.~Coulembier, P.~Etingof, V.~Ostrik, B.~Pauwels: Monoidal abelian envelopes with a quotient property. J. Reine Angew. Math. 794 (2023), 179--214. 

\bibitem[CPS]{CPS}  E.~Cline, B.~Parshall, L.~Scott: Induced modules and affine quotients. Math. Ann. 230 (1977), no. 1, 1--14.

\bibitem[DMNO]{DMNO}  A.~Davydov, M.~M\"uger,D.~Nikshych, V.~Ostrik: The Witt group of non-degenerate braided fusion categories. J. Reine Angew. Math. 677 (2013), 135--177.
	
	\bibitem[De1]{Del90} P.~Deligne: Cat\'egories tannakiennes. The Grothendieck Festschrift, Vol. II, 111--195, Progr. Math., 87, Birkh\"auser Boston, Boston, MA, 1990. 
	
	\bibitem[De2]{Del02} P.~Deligne: Cat\'egories tensorielles. Mosc. Math. J. 2 (2002), no. 2, 227--248.
	
	\bibitem[De3]{Del07} P.~Deligne:
La cat\'egorie des repr\'esentations du groupe sym\'etrique $S_t$, lorsque t n'est pas un entier naturel. Algebraic groups and homogeneous spaces, 209--273,
Tata Inst. Fund. Res. Stud. Math., 19, Tata Inst. Fund. Res., Mumbai, 2007.


 \bibitem[DM]{DM} P.~Deligne, J.S.~Milne: Tannakian Categories. In
Hodge cycles, motives, and Shimura varieties. 
Lecture Notes in Mathematics, 900. Springer-Verlag, Berlin-New York, 1982, pp. 101-228.
	
	
\bibitem[EHO]{EHO} P.~Etingof, N.~Harman, V.~Ostrik:
$p$-adic dimensions in symmetric tensor categories in characteristic $p$.
Quantum Topol. 9 (2018), no. 1, 119--140.	
	
	\bibitem[EHS]{EHS} I.~Entova-Aizenbud, V.~Hinich, V.~Serganova:
Deligne categories and the limit of categories $Rep(GL(m|n))$.
Int. Math. Res. Not. IMRN 2020, no. 15, 4602--4666. 


\bibitem[EG1]{EG1} P.~Etingof,S.~Gelaki: Finite symmetric integral tensor categories with the Chevalley property, with an appendix by Kevin Coulembier and Pavel Etingof. Int. Math. Res. Not. IMRN 2021, no. 12, 9083--9121.

 \bibitem[EG2]{EG2} P.~Etingof,S.~Gelaki: Finite symmetric tensor categories with the Chevalley property in characteristic 2. J. Algebra Appl. 20 (2021), no. 1, Paper No. 2140010.
	

\bibitem[EGNO]{EGNO}P.~Etingof, S.~Gelaki, D.~Nikshych, V.~Ostrik:
Tensor categories. 
Mathematical Surveys and Monographs, 205. American Mathematical Society, Providence, RI, 2015. 

\bibitem[EO]{EOf} P. Etingof, V. Ostrik: On the Frobenius functor for symmetric tensor categories in positive characteristic.  J. Reine Angew. Math. 773 (2021), 165--198.


\bibitem[Ha]{Ha} N.~Harman: Stability and periodicity in the modular
representation theory of symmetric groups. arXiv:1509.06414.

\bibitem[HS]{HS} N.~Harman, A.~Snowden: Discrete pre-Tannakian categories.  arXiv:2304.05375.

\bibitem[IKO]{IKO} M.S.~Im, M.~Khovanov, V.~Ostrik: Universal construction in monoidal and non-monoidal settings, the Brauer envelope, and pseudocharacters. arXiv:2303.02696.

\bibitem[Os]{Os} V.~Ostrik: On symmetric fusion categories in positive characteristic. Selecta Math. (N.S.) 26 (2020), no. 3, Paper No. 36, 19 pp.
 
 


\bibitem[Ta]{Ta} M.~Takeuchi:
Finite Hopf algebras in braided tensor categories.
J. Pure Appl. Algebra 138 (1999), no. 1, 59--82.

\bibitem[Ve]{Ve}
S.~Venkatesh:
Hilbert basis theorem and finite generation of invariants in symmetric tensor categories in positive characteristic. 
Int. Math. Res. Not. IMRN 2016, no. 16, 5106--5133. 
 	\end{thebibliography}
\end{document}